\documentclass[11pt, twoside, leqno]{article}

\usepackage{amssymb}
\usepackage{amsmath}
\usepackage{amsthm}
\usepackage{color}
\usepackage{mathrsfs}

\usepackage{indentfirst}

\usepackage{txfonts}

\allowdisplaybreaks

\def\rr{{\mathbb R}}

\def\ls{\lesssim}
\def\gs{\gtrsim}

\def\XXint#1#2#3{{\setbox0=\hbox{$#1{#2#3}{\int}$ }
\vcenter{\hbox{$#2#3$ }}\kern-.6\wd0}}

\def\({\left(}
\def \){ \right)}

\newtheorem{theorem}{Theorem}[section]
\newtheorem{lemma}[theorem]{Lemma}

\newtheorem{proposition}[theorem]{Proposition}

\theoremstyle{definition}
\newtheorem{remark}[theorem]{Remark}

\renewcommand{\appendix}{\par
   \setcounter{section}{0}%
   \setcounter{subsection}{0}%
   \setcounter{subsubsection}{0}%
   \gdef\thesection{\@Alph\c@section}%
   \gdef\thesubsection{\@Alph\c@section.\@arabic\c@subsection}%
   \gdef\theHsection{\@Alph\c@section.}%
   \gdef\theHsubsection{\@Alph\c@section.\@arabic\c@subsection}%
   \csname appendixmore\endcsname
 }

\numberwithin{equation}{section}

\begin{document}

\arraycolsep=1pt

\title{\bf\Large $L^p$ Boundedness of Hilbert Transforms Associated with Variable Plane Curves
\footnotetext{\hspace{-0.35cm} 2010 {\it
Mathematics Subject Classification}. Primary 42B20;
Secondary 42B25.
\endgraf {\it Key words and phrases.} Hilbert transform, Carleson operator,
Littlewood-Paley operator, shifted maximal operator, variable plane curve.
\endgraf This work was partially supported by
 NSFC-DFG  (Grant Nos.
11761131002).}}
\author{Haixia Yu and Junfeng Li\footnote{Corresponding author.}}
\date{}
\maketitle

\vspace{-0.7cm}

\begin{center}
\begin{minipage}{13cm}
{\small {\bf Abstract}\quad Let $p\in (1,\infty)$. In this paper, for any given measurable function
$u:\ \mathbb{R}\rightarrow \mathbb{R}$ and a generalized plane curve $\gamma$ satisfying some conditions,
the $L^p(\mathbb{R}^2)$ boundedness of the Hilbert transform along the variable plane curve $u(x_1)\gamma$
$$H_{u,\gamma}f(x_1,x_2):=\mathrm{p.\,v.}\int_{-\infty}^{\infty}f(x_1-t,x_2-u(x_1)\gamma(t))
\,\frac{\textrm{d}t}{t}, \quad \forall\, (x_1,x_2)\in\mathbb{R}^2, $$ is obtained.
At the same time, the $L^p(\mathbb{R})$ boundedness of the corresponding
Carleson operator along the general curve $\gamma$
$$\mathcal{C}_{u,\gamma}f(x):=\mathrm{p.\,v.}\int_{-\infty}^{\infty}e^{iu(x)\gamma (t)}f(x-t)\,\frac{\textrm{d}t}{t},
\quad\forall\, x\in\mathbb{R}, $$
is also obtained. Moreover, all the bounds are independent of the measurable function $u$.
}
\end{minipage}
\end{center}


\section{Introduction}

Let $u:\ \mathbb{R}\rightarrow \mathbb{R}$ be a measurable function and $\gamma$ be a generalized plane curve, the \emph{Hilbert transform $H_{u,\gamma}$} along the variable plane curve $u(x_1)\gamma$ is defined by setting, for any function $f$ in the Schwartz class $\mathcal{S}(\mathbb{R}^2)$ and $(x_1,x_2)\in\mathbb{R}^2$,
\begin{equation}\label{Hilbert transform}H_{u,\gamma}f(x_1,x_2):=\mathrm{p.\,v.}\int_{-\infty}^{\infty}f(x_1-t,x_2-u(x_1)\gamma(t))\,\frac{\textrm{d}t}{t}.
\end{equation}
Here and hereafter, $\mathrm{p.\,v.}$ denotes the principal-value integral. The corresponding \emph{Carleson operator $\mathcal{C}_{u,\gamma}$} along the general curve $\gamma$ is defined by setting, for any $f\in\mathcal{S}(\mathbb{R})$ and $x\in\mathbb{R}$,
\begin{equation}\label{Carleson operator}\mathcal{C}_{u,\gamma}f(x):= \mathrm{p.\,v.}\int_{-\infty}^{\infty}e^{iu(x)\gamma (t)}  f(x-t)\,\frac{\textrm{d}t}{t}.\end{equation}

Let $p\in (1,\infty)$. In this paper we pursue the $L^p$ boundedness
of \eqref{Hilbert transform} and \eqref{Carleson operator} for some general
plane curves $\gamma$. We first state our main results and make some remarks
and then give the motivations. For the Hilbert transform \eqref{Hilbert transform},
we have the following result.

\begin{theorem}\label{theorem 1.1}
Let $u:\ \mathbb{R}\rightarrow \mathbb{R}$ be a measurable function, $\gamma\in C^{3}(\mathbb{R})$ be either odd or even, with $\gamma(0)=\gamma'(0)=0$, and convex on $(0,\infty)$, satisfying
\begin{enumerate}
  \item[\rm(i)] $\frac{\gamma'(2t)}{\gamma'(t)}$ is decreasing and bounded by a constant $C_1$ from above on $(0,\infty)$,
  \item[\rm(ii)] there exists positive constant $C_2$ such that $\frac{t\gamma''(t)}{\gamma'(t)}\leq C_2$ on $(0,\infty)$,
  \item[\rm(iii)] there exists a positive constant $C_3$ such that $|(\frac{\gamma''}{\gamma'})'(t)|\geq \frac{C_3}{t^2}$ on $(0,\infty)$,
  \item[\rm(iv)] $\frac{\gamma'''(t)}{\gamma''(t)}$ is strictly monotone or equals to a constant on $(0,\infty)$.
\end{enumerate}
Then, for any given $p\in(1,\infty)$, there exists a positive constant $C$ such that, for any $f\in L^{p}(\mathbb{R}^{2})$,
$$\|H_{u,\gamma}f\|_{L^{p}(\mathbb{R}^{2})}\leq C \|f\|_{L^{p}(\mathbb{R}^{2})}$$
and, moreover, the bound $C$ is independent of $u$.
\end{theorem}

For the Carleson operator \eqref{Carleson operator}, we also have the following boundedness.

\begin{theorem}\label{theorem 1.2}
Let $u$ and $\gamma$ be the same as in Theorem \ref{theorem 1.1}, and $p\in (1,\infty)$. Then there exists a positive constant $C$ such that, for any $f\in L^{p}(\mathbb{R})$,
$$\|\mathcal{C}_{u,\gamma}f\|_{L^{p}(\mathbb{R})}\leq C \|f\|_{L^{p}(\mathbb{R})}$$
and, moreover, the bound $C$ is independent of $u$.
\end{theorem}

Throughout this paper, we use $C$ to denote a \emph{positive
constant} that is independent of the main parameters involved, but whose
value may vary from line to line. The \emph{positive constants with subscripts},
such as $C_1$ and $C_2$, are the same in different
occurrences. For two real functions $f$ and $g$, we use $f\ls g$ or $g\gs f$ to denote $f\le Cg$ and,
if $f\ls g\ls f$, we then write $f\approx g$.

\begin{remark}\label{remark 1.1}
Since $\gamma\in C^{3}(\mathbb{R})$, $\gamma(0)=\gamma'(0)=0$ and $\gamma$ is convex on $(0,\infty)$, it implies that
$\gamma''\geq 0$ on $(0,\infty)$,
and $\gamma'$ is increasing and $\gamma'\geq0$ on $(0,\infty)$. Thus, $\gamma\geq0$ on $(0,\infty)$. We also know that $\gamma'(t)\geq \gamma'(1)$ for any $t\in[1,\infty)$, which further follows that
$\lim_{t\rightarrow \infty}\gamma(t)=\infty.$
Since $\gamma'$ is increasing on $(0,\infty)$ and $\gamma(0)=0$, it is easy to check that $1\leq\frac{t\gamma'(t)}{\gamma(t)}$ for any $t\in (0,\infty)$. On the other hand, since $\gamma(0)=0$, by the Cauchy mean value theorem, for any given $t\in(0,\infty)$ there exists $\xi_t\in (0,t)$ such that
$$\frac{t\gamma'(t)}{\gamma(t)}=\frac{t\gamma'(t)-0\gamma'(0)}{\gamma(t)-\gamma(0)}=\frac{\gamma'(\xi_t)+\xi_t\gamma''(\xi_t)}{\gamma'(\xi_t)}.$$
Thus, by Theorem \ref{theorem 1.1}(ii), there exists $C_4:=C_2+1$ such that
$$1\leq \frac{t\gamma'(t)}{\gamma(t)}\leq C_4, \quad \forall\, t\in (0,\infty).$$
\end{remark}

\begin{remark}\label{remark 1.2}
Since $\gamma'$ is increasing on $(0,\infty)$, with Theorem \ref{theorem 1.1}(i), we always have
$$1\leq \frac{\gamma'(2t)}{\gamma'(t)}\leq C_ 1, \quad \forall\, t\in (0,\infty).$$
\end{remark}

\begin{remark}\label{remark 1.3}
Follows are some curves $\gamma$ satisfying all the conditions of Theorem \ref{theorem 1.1}. We here only write the part for any $t\in [0,\infty)$. For any $t\in (-\infty,0]$, it is given by its even or odd property. For example,
\begin{enumerate}
  \item[\rm(i)] for any $t\in [0,\infty)$, $\gamma(t):=t^\alpha, \alpha\in(1,\infty)$,
  \item[\rm(ii)] for any $t\in [0,\infty)$, $\gamma(t):=t^2\log(1+t)$,
  \item[\rm(iii)] for any $t\in [0,\infty)$, $\gamma(t):=\int_0^t \tau^\alpha\log(1+\tau)\, \textrm{d}\tau$, $\alpha\in(1,\infty)$.
\end{enumerate}
\end{remark}

\begin{remark}\label{remark 1.0} Theorem \ref{theorem 1.1}(iv) is introduced to guarantee Lemma \ref{lemma 2.3}, which will be used to establish the crucial estimate of the oscillatory integral in Proposition \ref{proposition 3.1}. From the proof of Lemma \ref{lemma 2.3}, Theorem \ref{theorem 1.1}(iv) can be replaced by a wider condition:
For any $c\neq 0$, the equation $\frac{\gamma'''(t-c)}{\gamma''(t-c)}=\frac{\gamma'''(t)}{\gamma''(t)}$ on $t\in\mathbb{R}$ and $t\neq0$, $t\neq c$,
has a finite number of solutions including there is no solution, or there is at most a finite number of intervals such that the equation above is established on each of the intervals, or both, where the number is independent of $c$.
\end{remark}

\begin{remark}\label{remark 1.4}
In (\cite{G2}, Theorem 1.2), Guo et al. obtained the $L^p(\mathbb{R}^2)$ boundedness of $H_{u,\gamma}$ with the curve as in Remark \ref{remark 1.3}(i) for any given $p\in(1,\infty)$, but with $\alpha\in(0,\infty)$ and $\alpha\neq 1$. Thus, as a special case, Theorem \ref{theorem 1.1} covers (\cite{G2}, Theorem 1.2) whenever $\alpha\in(1,\infty)$. The work \cite{G2} illuminate us a lot in the proofs but we still make several contributions in the argument.  For the homogeneous curve, it is easy to see that $\gamma(ab)=\gamma(a)\gamma(b)$ for any $a,b\in(0,\infty)$. Since we seek for the $L^p(\mathbb{R}^n)$ boundedness of $H_{u,\gamma}$ with the bound independent of $u$, it is nature to absorb $u(x)$ by $\gamma$ for any fixed $x$. This can be easily obtained with $\gamma(t):=t^\alpha$ since $$|u(x)|\gamma(t)=\gamma(|u(x)|^{\frac{1}{\alpha}}t).$$ This property of course can not be hold by a general curve $\gamma$ and this property is crucial to make further decomposition. Motivated by \cite{G8} we introduce the map $n:\mathbb{R}\rightarrow\mathbb{Z}$ for any $x\in\mathbb{R}$ such that
$$\frac{1}{\gamma(2^{n(x)+1})}\leq|u(x)|\leq\frac{1}{\gamma(2^{n(x)})}.$$
This formula first appears in \eqref{eq:3.1}.

Another difficulty appearing in the $L^2(\mathbb{R}^2)$ estimations of the Hilbert transform $H_{u,\gamma}$ and the $L^p(\mathbb{R})$ boundedness of the Carleson operator $\mathcal{C}_{u,\gamma}$ for any given $p\in(1,\infty)$. It is crucial to establish a decay estimate of an oscillatory integral as in Propostion \ref{proposition 3.1}. If we have a homogeneous curve as in Remark \ref{remark 1.3}(i), it is easy to calculate the derivatives of the phase functions and the decay estimation would be easier to obtain. But for general curve $\gamma$ we need more complicated analysis and the assumptions (i),(ii),(iii) and (iv) of Theorem \ref{theorem 1.1} on the curves appear naturally during the estimation.

The main difficulty appearing in the $L^p(\mathbb{R}^2)$ estimations of the Hilbert transform $H_{u,\gamma}$ for any given $p\in(1,\infty)$ is as follows. By the Littlewood-Paley theory and notice the commutation relation $H_{u,\gamma}P_l=P_lH_{u,\gamma}$ for any $l\in \mathbb{Z}$, we need to establish a refined estimate for $H_{u,\gamma,k+n_l(x_1)}P_l$ by the shifted maximal operator. Here $P_l$ denotes the Littlewood-Paley decompostion operator according to the second variable and $l\in \mathbb{Z}$. Guo et al. in \cite{G2} considered the homogeneous case as in Remark \ref{remark 1.3}(i) did not need $n_l(x_1)$, where the map $n_l:\ \mathbb{R}\rightarrow \mathbb{Z}$ for any $x_1\in\mathbb{R}$ and $l\in\mathbb{Z}$ defined by
$$\frac{1}{\gamma(2^{n_l(x_1)+1})}\leq 2^l|u(x_1)|\leq\frac{1}{\gamma(2^{n_l(x_1)})},$$
see \eqref{eq:y3}. This new note allows us to obtain the refined estimate for $H_{u,\gamma,k+n_l(x_1)}P_l$ with a great effort to control the dyadic pieces by the shifted maximal operator. This is the main difficulty we overcomed and appearing in the estimations of \eqref{eq:3.62}.
\end{remark}

First, if $u:\ \mathbb{R}\rightarrow \mathbb{R}$ is a real number $\lambda$, then the operator in \eqref{Hilbert transform} is equivalent to the following \emph{directional Hilbert transform $H_{\lambda,\gamma}$} along a general curve $\gamma$ defined for a fixed direction $(1,\lambda)$ as
$$H_{\lambda,\gamma}f(x_1,x_2):=\mathrm{p.\,v.}\int_{-\infty}^{\infty}f(x_1-t,x_2-\lambda \gamma(t))\,\frac{\textrm{d}t}{t},\quad\forall\, (x_1,x_2)\in\mathbb{R}^2,$$
whose $L^p(\mathbb{R}^2)$ boundedness can be obtained obviously by the \emph{Hilbert transform $H_{\gamma}$} along a general curve $\gamma$:
\begin{equation}\label{H2}
H_{\gamma}f(x_1,x_2):=\mathrm{p.\,v.}\int_{-\infty}^{\infty}f(x_1-t,x_2-\gamma(t))
\,\frac{\textrm{d}t}{t},\quad\forall\, (x_1,x_2)\in\mathbb{R}^2.\end{equation}
This operator has independent interests, which is one of the motivations of this paper. There are enumerate literatures on this problem; see, for example, \cite{CCCD,CCVWW,CVWW,NVWW,VWW,W}. On the other hand, letting $p\in (1,\infty)$, it is not hard to obtained that
$$\sup_{\lambda\in \mathbb{R}}\left\|    H_{\lambda,\gamma}f  \right\|_{L^{p}(\mathbb{R}^{2})}\leq C \|f\|_{L^{p}(\mathbb{R}^{2})}.$$
But the $L^p(\mathbb{R}^2)$ boundedness of the corresponding maximal operator $\sup_{\lambda\in \mathbb{R}}|H_{\lambda,\gamma}f(x_1,x_2)|$ might not be obtained so obviously. In fact, by linearization, this uniformity estimate is tantamount to the $L^p(\mathbb{R}^2)$ estimate for
$$H_{U,\gamma}f(x_1,x_2):=\mathrm{p.\,v.}\int_{-\infty}^{\infty}f(x_1-t,x_2-U(x_1,x_2) \gamma(t))\,\frac{\textrm{d}t}{t},\quad\forall\, (x_1,x_2)\in\mathbb{R}^2,$$
and the bound must be independent of the measurable function $U$. But, it is well known that $H_{U,\gamma}$ might not lie in any $L^p(\mathbb{R}^2)$ if we only assume $U$ is a measurable function, see \cite{G2}. Therefore, we cannot hope to get that
$$\left\|  \sup_{\lambda\in \mathbb{R}} | H_{\lambda,\gamma}f | \right\|_{L^{p}(\mathbb{R}^{2})}\leq C \left\|f\right\|_{L^{p}(\mathbb{R}^{2})}$$
for any given $p\in (1,\infty)$. Instead of this, Theorem \ref{theorem 1.1} shows that
$$\left\| \sup_{\lambda\in \mathbb{R}}\left\| H_{\lambda,\gamma}f(\cdot_1,\cdot_2) \right \|_{L^{p}(\mathbb{R}^1_{x_2})}\right \|_{L^{p}(\mathbb{R}^1_{x_1})}\leq C \left\|f\right\|_{L^{p}(\mathbb{R}^{2})}$$
for all $p\in (1,\infty)$, which squeezes the supremum between the two $L^p$ norms on the left hand side. Here and hereafter, $\cdot_1$ and $\cdot_2$ denote the first variable $x_1$ and the second variable $x_2$, respectively. As Stein and Wainger pointed out in \cite{SW1} that the curvature of the considered curve plays a crucial role in this project, the conditions $(\textrm{i})$, $(\textrm{ii})$, $(\textrm{iii})$ and $(\textrm{iv})$ of Theorems \ref{theorem 1.1} are used to describe the curvature of the considered curve $\gamma$.

Second, if $\gamma(t):=t$ for any $t\in \mathbb{R}$, Bateman in \cite{B1} proved that $H_{u,\gamma}P_k$ is bounded on $L^p(\mathbb{R}^2)$ for any given $p\in(1,\infty)$ uniformly for any $k\in\mathbb Z$, where $P_k$ denotes the Littlewood-Paley projection operator in the second variable. Later, Bateman and Thiele in \cite{B2} proved the $L^p(\mathbb{R}^2)$ boundedness of $H_{u,\gamma}$ for all $p\in(\frac{3}{2},\infty)$. Moreover, let $\gamma$ be $|t|^\alpha$ or $\textrm{sgn}(t)|t|^\alpha$ for any $t\in \mathbb{R}$, $\alpha\in(0,\infty)$, $\alpha\neq1$, Guo et al. in \cite{G2} obtained the $L^p(\mathbb{R}^2)$ boundedness of $H_{u,\gamma}$ for any given $p\in(1,\infty)$.
Furthermore, Carbery et al. in \cite{CWW} obtained the $L^p(\mathbb{R}^2)$ boundedness of $H_{u,\gamma}$ for any given $p\in(1,\infty)$, but with the restriction that $u(x_1):=x_1$ for any $x_1\in \mathbb{R}$, where $\gamma\in C^{3}(\mathbb{R})$ is either odd or even, convex curve on $(0,\infty)$, and satisfies $\gamma(0)=\gamma'(0)=0$ and the quantity $\frac{t\gamma''(t)}{\gamma'(t)}$ is decreasing and bounded below on $(0,\infty)$. Under the same condition, Bennett in \cite{BJ} obtained the $L^2(\mathbb{R}^2)$ boundedness of
\begin{equation}\label{HP}H_{P,\gamma}f(x_1,x_2):=\mathrm{p.\,v.}\int_{-\infty}^{\infty}f(x_1-t,x_2-P(x_1)\gamma(t))\,\frac{\textrm{d}t}{t}, \quad\forall\, (x_1,x_2)\in\mathbb{R}^2, \end{equation}
for any general polynomial $P$. More recently, Chen and Zhu in \cite{CZx} obtained the $L^2(\mathbb{R}^2)$ boundedness of $H_{P,\gamma}$ in \eqref{HP} by asking the curvature condition as
\begin{center}
$(\frac{\gamma''}{\gamma'})'(t)\leq -\frac{\lambda_1}{t^2}$ for any $t\in[0,\infty)$ and some positive constant $\lambda_1$.
\end{center}
In \cite{YL}, we also obtained the $L^2(\mathbb{R}^2)$ boundedness of $H_{P,\gamma}$ in \eqref{HP} if the curvature condition for $\gamma\in C^{2}(\mathbb{R})$ is replaced by
\begin{enumerate}
  \item[\rm(i)] $\frac{\gamma''(t)}{\gamma'(t)}$ is decreasing on $(0,\infty)$,
  \item[\rm(ii)] there exists a positive constant $\lambda_2$ such that $\frac{t\gamma''(t)}{\gamma'(t)}\geq \lambda_2$ for any $t\in (0,\infty)$,
  \item[\rm(iii)] $\gamma''(t)$ is monotone on $(0,\infty)$.
\end{enumerate}
All of these results are based on iteration on the degree of polynomial $P$ and hence can not extend to general measurable function $u$. Thus, Theorem \ref{theorem 1.1} is the first result on the generalized plane curve $\gamma$.

The Carleson operator \eqref{Carleson operator} along a plane curve appears naturally in the study of the $L^2(\mathbb{R}^2)$ boundedness of the Hilbert transform \eqref{Hilbert transform}. This fact will be stated in Section 3. This operator itself is also interesting.
The original \emph{Carleson operator $\mathcal{C}$} is defined by setting, for any
$f\in\mathcal{S}(\mathbb{R})$ and $x\in\mathbb{R}$,
$$\mathcal{C}f(x):=\sup_{N\in \mathbb{R}} \left| \mathrm{p.\,v.}\int_{-\infty}^{\infty}e^{iNy}  f(x-y)\,\frac{\textrm{d}y}{y}\right|.$$
By linearization, the estimate
$$\|\mathcal{C}f\|_{L^p(\mathbb{R})}\leq C \|f\|_{L^p(\mathbb{R})}$$
is equivalent to
$$\|\mathcal{C}_{u}f\|_{L^p(\mathbb{R})}\leq C \|f\|_{L^p(\mathbb{R})},$$
where $u:\ \rr\to\rr$ is a measurable function,
$$\mathcal{C}_{u}f(x):=\mathrm{p.\,v.}\int_{-\infty}^{\infty}e^{iu(x)t}  f(x-t)\,\frac{\textrm{d}t}{t}, \quad\forall\, x\in\mathbb{R},$$
and the bound $C$ is a positive constant independent of $u$. In \cite{C}, Carleson obtained the $L^2(\mathbb{R})$ boundedness of $\mathcal{C}$, which plays an important role in obtaining almost everywhere convergence of Fourier series of $L^2(\mathbb{R})$ functions and also confirmed the famous Luzin conjecture. Hunt later obtained its $L^p(\mathbb{R})$ boundedness for any given $p\in(1,\infty)$ in \cite{H}. For further results about $\mathcal{C}$, we refer the reader to \cite{F,LT,PS1}. Stein and Wainger in \cite{SW} extended $\mathcal{C}_{u}$ to the \emph{Carleson operator $\mathcal{C}_{u,d}$} along a homogeneous curve $t^d$ with integer $d>1$, namely, for any
$f\in\mathcal{S}(\mathbb{R})$ and $x\in\mathbb{R}$,
$$\mathcal{C}_{u,d}f(x):=  \mathrm{p.\,v.}\int_{-\infty}^{\infty}e^{iu(x)t^d}  f(x-t)\,\frac{\textrm{d}t}{t}.$$
Stein and Wainger showed that its $L^p(\mathbb{R})$ bound is independent of $u$, where $p\in(1,\infty)$. Guo in \cite{G7} extended $\mathcal{C}_{u,d}$ further into the Carleson operator along a homogeneous curve $|t|^{\varepsilon_1}$
or $\textrm{sgn}(t) |t|^{\varepsilon_2}$, where $\varepsilon_1, \varepsilon_2\in\mathbb{R} $, $\varepsilon_1\neq 1$ and $\varepsilon_2\neq 0$. Thus, it is natural to consider the Carleson operator $\mathcal{C}_{u,\gamma}$ along a more general curve in \eqref{Carleson operator}. This is one of our main results to establish its $L^p(\mathbb{R})$ boundedness of Carleson operator along a general curve $\gamma $ in Theorem \ref{theorem 1.2} with the bound independent of $u$.

This paper is organized as follows. In Section \ref{section 2.1}, we provide some lemmas serve as a preparation for the corresponding proof of Theorem \ref{theorem 1.2}. In Section \ref{section 2.2} we give the proof of Theorem \ref{theorem 1.2}. It then establishes the $L^2(\mathbb{R}^2)$ boundedness of \eqref{Hilbert transform}. Section \ref{section 3} is devoted to obtaining the single annulus $L^p(\mathbb{R}^2)$ estimate for \eqref{Hilbert transform} for any given $p\in(1,\infty)$, it is Theorem \ref{theorem 1.3}. Section \ref{section 4} is devoted to obtaining the $L^p(\mathbb{R}^2)$ boundedness of \eqref{Hilbert transform} for any given $p\in(1,\infty)$ which then helps us to finish the proof of Theorem \ref{theorem 1.1}.


\section{Proof of Theorem \ref{theorem 1.2}}\label{proof of 1.2}

\subsection{Some lemmas}\label{section 2.1}

Before giving the proof of Theorem \ref{theorem 1.2}, we state three lemmas. Van der Corput's lemma is a useful tool to bound an oscillatory integral but, for the case $k=1$, a simple lower bound on $|\phi'|$ is not sufficient. We need to add a condition that $\phi'$ is monotonic such that $\int_a^b | \frac{\textrm{d}}{\textrm{d}t} (\frac{1}{\phi'(t)}) |\,\textrm{d}t$ is dominated by a constant. Lemma \ref{lemma 2.1} is a slight variant of van der Corput's lemma which replaced the additional condition by $\phi''$ is bounded from above. Lemma \ref{lemma 2.2} is used to get an interesting fact, for the phase function $\phi$ of the considered oscillatory integral, we must have $|\phi'|\geq C$ or $|\phi''|\geq C$. However, it is not sufficient to complete our estimate even if we obtained the surprising lower bound on $|\phi'|$ or $|\phi''|$, since we can take an infinite number of intervals such that the lower bound is established. Lemma \ref{lemma 2.3} is used to make sure that such case does not happen.

\begin{lemma}\label{lemma 2.1}
Suppose $\phi$ is real-valued and smooth in $(a,b)$, and that both $|\phi'(x)|\geq \sigma_1$ and $|\phi''(x)|\leq \sigma_2$ for any $x\in (a,b)$. Then
$$\left|\int_a^b e^{i\phi (t)}\,\textrm{d}t\right|\leq \frac{2}{\sigma_1}+(b-a)\frac{\sigma_2}{\sigma_1^2}.$$
\end{lemma}

\begin{proof}
From the proof of the van der Corput lemma's, see, for example, (\cite{S}, P.332, Proposition 2), which bounds the integral by
$$\left| \frac{e^{i\phi (b)}}{i\phi'(b) } -\frac{e^{i\phi (a)}}{i\phi'(a) }\right|+ \int_a^b \left| \frac{\textrm{d}}{\textrm{d}t} \left(\frac{1}{\phi'(t)}\right)  \right|\,\textrm{d}t\lesssim \frac{2}{\sigma_1}+\int_a^b \left| \frac{\phi''(t)}{\phi'(t)^2} \right|\textrm{d}t  \lesssim \frac{2}{\sigma_1}+ (b-a)\frac{\sigma_2}{\sigma_1^2},$$
it is easy to deduce the desired conclusion of Lemma \ref{lemma 2.1}.
\end{proof}

\begin{lemma}\label{lemma 2.2}
 (\cite{G8}, Lemma 4.5) Let $A$ be an invertible $n\times n$ matrix and $x\in \mathbb{R}^n$. Then
$$|Ax|\geq |\textrm{det} A| \|A\|^{1-n} |x|,$$
where $\|A\|$ denotes the matrix norm $\sup_{|x|=1}|Ax|$.
\end{lemma}

\begin{lemma}\label{lemma 2.3}
Let $\gamma$ be the same as in Theorem \ref{theorem 1.1}, for any $a,b,c,d\in \mathbb{R}$ and $d>0$, there is at most a finite number of intervals such that
\begin{align}\label{eq:hx}
|a\gamma'(t)-b\gamma'(t-c)|>d
\end{align}
is established on each of the intervals, where $t\in \mathbb{R}$ and the number of intervals is independent of $a,b,c,d$.
\end{lemma}

\begin{proof}
Since \eqref{eq:hx} is equivalent to
$$a\gamma'(t)-b\gamma'(t-c)-d>0$$
or
$$a\gamma'(t)-b\gamma'(t-c)+d<0.$$
Note that $\gamma\in C^{3}(\mathbb{R})$, it is enough to show that
\begin{align}\label{eq:hx1}
a\gamma''(t)-b\gamma''(t-c)=0
\end{align}
has a finite number of solutions including there is no solution, or there is at most a finite number of intervals such that \eqref{eq:hx1} is established on each of the intervals, or both, where the number is independent of $a,b,c$. There are some cases:

If $b=0$ and $a=0$. Then \eqref{eq:hx} does not exist, in other words, there is no intervals such that \eqref{eq:hx} is established.

If $b=0$ and $a\neq0$. Since $\gamma$ is either odd or even and $\gamma'$ is increasing on $(0,\infty)$, then the Lemma \ref{lemma 2.3} is obtained obviously.

If $b\neq0$, $c=0$ and $a=b$. Then \eqref{eq:hx} does not exist.

If $b\neq0$, $c=0$ and $a\neq b$. Then \eqref{eq:hx} is equivalent to $|(a-b)\gamma'(t)|>d$, as what we had stated, it is easy to see that the Lemma \ref{lemma 2.3} is established.

If $b\neq0$ and $c\neq0$. From Theorem \ref{theorem 1.1}(iv), $\gamma''(t)\neq0$ for any $t\in(0,\infty)$, note that $\gamma$ is either odd or even, then $\gamma''(t)\neq0$ for any $t\in(-\infty,0)\bigcup (0,\infty)$. It is easy to see that we should only consider $t\neq0$ and $t\neq c$ for \eqref{eq:hx1}. Then \eqref{eq:hx1} is equivalent to
\begin{align}\label{eq:hx2}
\frac{a}{b}=\frac{\gamma''(t-c)}{\gamma''(t)}, \quad t\neq0, \ t\neq c, \ t\in \mathbb{R}.
\end{align}
Let
$$F_c(t):=\frac{\gamma''(t-c)}{\gamma''(t)}, \quad t\neq0, \ t\neq c, \ t\in \mathbb{R}.$$
We see that for any $t\neq0, t\neq c, t\in \mathbb{R}$,
\begin{align}\label{eq:hx3}
F_c'(t)=\frac{ \gamma'''(t-c)\gamma''(t)-\gamma''(t-c)\gamma'''(t)  }{(\gamma''(t))^2}=\frac{\gamma''(t-c) \left[\frac{\gamma'''(t-c)}{\gamma''(t-c)}-\frac{\gamma'''(t)}{\gamma''(t)} \right] }{\gamma''(t)}.
\end{align}
From Theorem \ref{theorem 1.1}(iv), $\frac{\gamma'''(t)}{\gamma''(t)}$ is strictly monotone or equals to a constant on $(0,\infty)$, since $\gamma$ is either odd or even, then the equation
\begin{align}\label{eq:hx4}
\frac{\gamma'''(t-c)}{\gamma''(t-c)}=\frac{\gamma'''(t)}{\gamma''(t)}, \quad t\neq0, \ t\neq c, \ t\in \mathbb{R},
\end{align}
has a finite number of solutions including there is no solution, or there is at most a finite number of intervals such that \eqref{eq:hx4} is established on each of the intervals, or both, where the number is independent of $c$. Therefore, $F_c'(t)$ in \eqref{eq:hx3} has the same character as \eqref{eq:hx4}. Then \eqref{eq:hx2} also has the same character as \eqref{eq:hx4}. This finishes the proof of Lemma \ref{lemma 2.3}.
\end{proof}

\subsection{$L^p(\mathbb{R})$ estimate for the Carleson operator $\mathcal{C}_{u,\gamma}$}\label{section 2.2}

We now show Theorem \ref{theorem 1.2}. The main strategy of our proof is to decompose our operator into a low frequency part and a high frequency part. We want to bound the low frequency part by some classical operators, such as the Hardy-Littlewood maximal operator and the maximal truncated Hilbert transform. For the high frequency part, which is further divided into a series of operators $\{S_k\}_{k=0}^\infty$. We want to get a decay estimate for each of $S_k$. The main tools is the $TT^*$ argument, the stationary phase method, and also these lemmas have been introduced in Section \ref{section 2.1}.

\begin{proof}[Proof of Theorem \ref{theorem 1.2}] Suppose smooth function
$\psi:\ \mathbb{R}\rightarrow\mathbb{R}$ is supported on $\left\{t\in\mathbb{R}:\ \frac{1}{2}\leq |t|\leq 2\right\}$
such that $0\leq \psi(t)\leq 1$ and $\Sigma_{l\in \mathbb{Z}} \psi_l(t)=1$ for any $t\neq 0$,
where $\psi_l(t):=\psi (2^{-l}t)$.  From Remark \ref{remark 1.1}, we have that $\gamma$ is increasing on $(0,\infty)$ and $\lim_{t\rightarrow \infty}\gamma(t)=\infty$. We can define $n:\ \mathbb{R}\rightarrow\mathbb{Z}$  such that, for any given $x\in \mathbb{R}$,
\begin{align}\label{eq:3.1}
\frac{1}{\gamma (2^{n(x)+1})}\leq |u(x)|\leq \frac{1}{\gamma (2^{n(x)})} .
\end{align}
For any given $x\in \mathbb{R}$, let
$$\mathcal{C}_{u,\gamma,k}f(x):= \int_{-\infty}^{\infty} e^{iu(x)\gamma (t)}  f(x-t)   \psi_k(t)\,  \frac{\textrm{d}t}{t},$$
and decompose
\begin{align}\label{eq:yly1}\mathcal{C}_{u,\gamma}f(x)
   =  \sum_{k\leq n(x)-1}   \mathcal{C}_{u,\gamma,k}f(x)+  \sum_{k\geq n(x)}  \mathcal{C}_{u,\gamma,k}f(x)
  =: \mathcal{C}_{u,\gamma}^{(1)}f(x)+ \mathcal{C}_{u,\gamma}^{(2)}f(x).
\end{align}

For the low frequency part $\mathcal{C}_{u,\gamma}^{(1)}f$, let $\sum_{k\leq n(x)-1} \psi_k(t)=:\phi (t)$, then
\begin{eqnarray*}
\mathcal{C}_{u,\gamma}^{(1)}f(x)
   &=&  \mathrm{p.\,v.} \int_{|t|\leq   2^{n(x)}}\left[e^{iu(x)\gamma (t)}  -1\right]f(x-t) \phi (t)\, \frac{\textrm{d}t}{t}+\mathrm{p.\,v.}\int_{|t|\leq   2^{n(x)}}f(x-t) \phi (t) \,\frac{\textrm{d}t}{t}\\
   &=:& T_1f(x)+ T_2f(x).
\end{eqnarray*}

For $T_1f$, since $\gamma'$ is increasing on $(0,\infty)$ and $\gamma(0)=0$, we have $\frac{\gamma(t)}{t}$ is increasing on $(0,\infty)$. This, combined with the fact that $\gamma$ is either odd or even and \eqref{eq:3.1}, further implies that
\begin{align}\label{eq:bnu1}
T_1f(x)
   &\leq   \int_{|t|\leq   2^{n(x)}} \left|f(x-t)  \right| \left|u(x)\right| \frac{\gamma(2^{n(x)})}{2^{n(x)}}   \phi (t)\,\textrm{d}t\\
    &\leq    \frac{1}{2^{n(x)}}   \int_{|t|\leq   2^{n(x)}} \left|f(x-t)  \right| \,\textrm{d}t  \lesssim  Mf(x)\nonumber.
\end{align}
Here and hereafter, $M$ denotes the \emph{Hardy-Littlewood maximal operator} defined by setting
$$Mf(x):=\sup_{r>0} \frac{1}{2r}  \int_{-r}^{r} \left|f(x-t)  \right|\, \textrm{d}t, \quad\forall\, x\in\mathbb{R}.$$

For $T_2f$, we have
\begin{align}\label{eq:bnu2}
|T_2f(x)|
   &=   \left|\int_{|t|\leq   2^{n(x)}}f(x-t) \frac{\phi (t)-1}{t}\, \textrm{d}t+ \mathrm{p.\,v.}\int_{|t|\leq   2^{n(x)}}f(x-t) \,\frac{\textrm{d}t}{t}\right|\\
   &\leq   \int_{2^{n(x)-1}\leq|t|\leq   2^{n(x)}} \left|f(x-t)  \right| \left|\frac{\phi (t)-1}{t} \right| \,\textrm{d}t+ \mathcal{H}^*f(x)\nonumber\\
    &\leq     \frac{1}{2^{n(x)-1}}   \int_{|t|\leq   2^{n(x)}} \left|f(x-t)  \right|\, \textrm{d}t+ \mathcal{H}^*f(x)\lesssim  Mf(x)+ \mathcal{H}^*f(x)\nonumber,
\end{align}
where $\mathcal{H}^*$ is the \emph{maximal truncated Hilbert transform}, which is defined by setting
$$\mathcal{H}^*f(x):=\sup_{\varepsilon,R>0} \left| \int_{\varepsilon<|t|<R} f(x-t) \, \frac{\textrm{d}t}{t}\right|, \quad\forall\, x\in\mathbb{R}.$$
Therefore, from \eqref{eq:bnu1} and \eqref{eq:bnu2} we have
$$\mathcal{C}_{u,\gamma}^{(1)}f(x)\lesssim  Mf(x)+ \mathcal{H}^*f(x).$$
It is well-known that both $M$ and $\mathcal{H}^*$ are bounded on $L^p(\mathbb{R})$, we conclude that
$$\|\mathcal{C}_{u,\gamma}^{(1)}f\|_{L^{p}(\mathbb{R})}\lesssim \|f\|_{L^{p}(\mathbb{R})},$$
where $p\in (1,\infty)$.

For the high frequency part $\mathcal{C}_{u,\gamma}^{(2)}f$. We can then write
$$\mathcal{C}_{u,\gamma}^{(2)}f(x)= \sum_{k\geq 0}  \int_{-\infty}^{\infty} e^{iu(x)\gamma (t)}  f(x-t)   \psi_{k+n(x)}(t) \, \frac{\textrm{d}t}{t}=: \sum_{k\geq 0} S_kf(x).$$
For any given $k\geq 0$,
\begin{eqnarray*}
 |S_kf(x)|
    &\leq&    \int_{2^{k+n(x)-1}\leq|t|\leq   2^{k+n(x)+1}} \left| f(x-t)\right|  \frac{\left| \psi_{k+n(x)}(t) \right|}{|t|}\, \textrm{d}t   \\
    &\leq&    \frac{1}{2^{k+n(x)-1}} \int_{|t|\leq   2^{k+n(x)+1}} \left| f(x-t)\right| \,  \textrm{d}t\lesssim   Mf(x).
\end{eqnarray*}
From this and the well-known $L^p(\mathbb{R})$ boundedness of $M$, we have that
\begin{align}\label{eq:3.2}
\|S_kf\|_{L^{p}(\mathbb{R})}\lesssim \|f\|_{L^{p}(\mathbb{R})},
\end{align}
and the bound depends only on $p$, where $p\in (1,\infty)$.
To make the summation over $k\geq0$, we need a decay estimate for $\|S_kf\|_{L^{p}(\mathbb{R})}$. To this aim, we claim that there exists a positive constant $\omega_0$ such that, for any $k\geq 0$,
\begin{align}\label{eq:3.3}
\|S_kf\|_{L^{2}(\mathbb{R})}\lesssim 2^{-\omega_0 k}\|f\|_{L^{2}(\mathbb{R})}.
\end{align}
Then, by interpolating between \eqref{eq:3.2} and \eqref{eq:3.3}, we obtain a positive constant $\omega_p$ such that
$$\|S_kf\|_{L^{p}(\mathbb{R})}\lesssim 2^{-\omega_p k}\|f\|_{L^{p}(\mathbb{R})}.$$
This allows us to sum up $k\geq 0$ and to obtain
$$\|\mathcal{C}_{u,\gamma}^{(2)}f\|_{L^{p}(\mathbb{R})}\lesssim \|f\|_{L^{p}(\mathbb{R})}$$
for any given $p\in (1,\infty)$.
Therefore it remains is to prove \eqref{eq:3.3}. We use the $TT^*$ argument which was introduced by Stein and Wainger in \cite{SW}. The dual operator of $S_k$ is given by
$$S_k^*g(y)=\mathrm{p.\,v.}\int_{-\infty}^{\infty} e^{-iu(z)\gamma (z-y)}    \psi(2^{-n(z)-k}(z-y)) \frac{g(z)}{z-y}\,  \textrm{d}z, \quad \forall\, y\in\mathbb{R}.$$
Therefore,
\begin{align}\label{eq:hx7}
& S_kS_k^*f(x)\\
   &\quad=   \mathrm{p.\,v.}\int_{-\infty}^{\infty} \mathrm{p.\,v.}\int_{-\infty}^{\infty} e^{-iu(z)\gamma (z-x+t)}  \frac{\psi(2^{-n(z)-k}(z-x+t))}{z-x+t}e^{iu(x)\gamma (t)}  \frac{\psi(2^{-n(x)-k}t)}{t} \, \textrm{d}t f(z)  \,\textrm{d}z.\nonumber
\end{align}
In the following calculation, without loss of generality, we may assume that $2^{n(x)}\leq 2^{n(z)}$. Let $\xi:=x-z$. Then the kernel of $S_kS_k^*$ can be written as
\begin{align}\label{eq:3.4}
 \mathrm{p.\,v.}\int_{-\infty}^{\infty} e^{-iu(z)\gamma (-\xi+t)}  \frac{\psi(2^{-n(z)-k}(-\xi+t))}{-\xi+t} e^{iu(x)\gamma (t)}  \frac{\psi(2^{-n(x)-k}t)}{t} \, \textrm{d}t.
\end{align}
We replace $2^{-n(x)-k}t$ by $t$,
\begin{align}\label{eq:3.5}
\mathrm{p.\,v.}\int_{-\infty}^{\infty} e^{-iu(z)\gamma (-\xi+2^{n(x)+k} t)}  \frac{\psi(    -\xi2^{-n(z)-k}+\frac{2^{n(x)}}{2^{n(z)}} t)}{-\xi+2^{n(x)+k} t} e^{iu(x)\gamma (2^{n(x)+k} t)}  \frac{\psi(t)}{t} \, \textrm{d}t.
\end{align}
Now, we further set $0<h:=\frac{2^{n(x)}}{2^{n(z)}}\leq1$ and $s:=\frac{\xi}{2^{n(z)+k}}$. Then the kernel becomes
\begin{align}\label{eq:3.7}
\frac{1}{2^{n(z)+k} }\,\mathrm{p.\,v.}\int_{-\infty}^{\infty} e^{iu(x)\gamma (2^{n(x)+k} t)-iu(z)\gamma \left(2^{n(z)+k}[ht-s]  \right)}  \frac{\psi\left( ht-s\right)}{ht-s}   \frac{\psi(t)}{t} \, \textrm{d}t.
\end{align}
To evaluate the above integral, we use a estimate from the following Proposition \ref{proposition 3.1}.
In fact, noticing $\frac{x-z}{2^{n(z)+k}}=s$, by \eqref{eq:3.8} of Proposition \ref{proposition 3.1}, we have therefore
\begin{eqnarray*}
 | S_kS_k^*f(x)|
    &=&\left| \int_{-\infty}^{\infty} \frac{1}{2^{n(z)+k} }\,\mathrm{p.\,v.}\int_{-\infty}^{\infty} e^{iu(x)\gamma (2^{n(x)+k} t)-iu(z)\gamma \left(2^{n(z)+k}[ ht-s ]  \right)}  \frac{\psi\left( ht-s\right)}{ht-s}   \frac{\psi(t)}{t} \, \textrm{d}t f(z) \, \textrm{d}z\right|\\
    &\lesssim&  \int_{-\infty}^{\infty} \frac{1}{2^{n(z)+k} }\left\{ \chi_{[-2^{-k r_1},2^{- k r_1}]}(s)+ 2^{-k r_2} \chi_{[-4,4]}(s) \right\}|f(z) |\, \textrm{d}z\\
     &\lesssim&  \frac{2^{-k r_1}}{2^{n(z)+k} 2^{-k r_1} }  \int_{\frac{|x-z|}{2^{n(z)+k}} \leq 2^{-k r_1} } |f(z) | \,\textrm{d}z+  \frac{2^{-k r_2}}{2^{n(z)+k}  }  \int_{\frac{|x-z|}{2^{n(z)+k}} \leq 4 } |f(z) | \,\textrm{d}z \\
     & \lesssim&   2^{-k r_1} Mf(x)+  2^{-k r_2} Mf(x)\lesssim   2^{-k r_0} Mf(x),
\end{eqnarray*}
where $\gamma_0:=\min\left\{r_1, r_2 \right\}$.
Since $M$ is bounded on $L^2(\mathbb{R})$ and hence
$$\|S_k\|_{L^2(\mathbb{R})\rightarrow L^2(\mathbb{R})}=\|S_kS_k^*\|^{\frac{1}{2}}_{L^2(\mathbb{R})\rightarrow L^2(\mathbb{R})}\lesssim 2^{- \frac{r_0}{2}k}.$$
This is \eqref{eq:3.3}, which completes the proof of Theorem \ref{theorem 1.2}.
\end{proof}

\begin{proposition}\label{proposition 3.1}
There exist positive constants $r_1$ and $r_2$ such that
\begin{align}\label{eq:3.8}
&\left|\mathrm{p.\,v.}\int_{-\infty}^{\infty} e^{iu(x)\gamma (2^{n(x)+k} t)-iu(z)\gamma \left(2^{n(z)+k}[ht-s]  \right)}  \frac{\psi\left( ht-s\right)}{ht-s}   \frac{\psi(t)}{t}  \,\textrm{d}t\right|\\
   &\quad\leq C \left\{ \chi_{[-2^{-k r_1},2^{- k r_1}]}(s)+ 2^{-k r_2} \chi_{[-4,4]}(s)\right\}\nonumber
\end{align}
for any $k\in \mathbb{N}$ and $x,z,s\in \mathbb{R}$, where $C$ is a positive constant independent of $k,x,z,s,u$.
\end{proposition}

\begin{proof}[Proof of Proposition \ref{proposition 3.1}] Since smooth function $\psi:\ \mathbb{R}\rightarrow\mathbb{R}$ is supported on $\left\{t\in\mathbb{R}:\ \frac{1}{2}\leq |t|\leq 2\right\}$ and $0<h\leq 1$, thus, $|t|\leq 2$, $|ht-s|\leq 2$ and $|s|\leq 4$.
Let
\begin{align}\label{eq:3.9}
Q(t):=u(x)\gamma (2^{n(x)+k} t)-u(z)\gamma (2^{n(z)+k}[ht-s] ), \quad \forall\, t\in\mathbb{R}.
\end{align}
It is clear that
\begin{align}\label{eq:3.10}
Q'(t)=u(x)2^{n(x)+k}\gamma' (2^{n(x)+k} t)-u(z)2^{n(z)+k} \gamma '(2^{n(z)+k}[ ht-s]  )h, \quad \forall\, t\in\mathbb{R},
\end{align}
and
\begin{align}\label{eq:3.11}
Q''(t)=u(x)2^{2(n(x)+k)}\gamma'' (2^{n(x)+k} t)-u(z)2^{2(n(z)+k)} \gamma''(2^{n(z)+k}[ ht-s]  )h^2, \quad \forall\, t\in\mathbb{R}.
\end{align}
To use lemmas \ref{lemma 2.1}, \ref{lemma 2.2} and \ref{lemma 2.3}, we need some estimates on $Q'$ and $Q''$. For this aim we consider two cases. We want to remind the reader the constants $C_1$ through $C_4$ are the same constants as in Theorem \ref{theorem 1.1} and Remark \ref{remark 1.1}.

{\bf Case~A}  $0<h\leq \frac{1}{4 C_1^3C_4}$.

Since $\frac{\gamma'(2t)}{\gamma'(t)}$ is decreasing on $(0,\infty)$, it follows that $\frac{\gamma'(2^kt)}{\gamma'(t)}=\frac{\gamma'(2^kt)}{\gamma'(2^{k-1}t)}\frac{\gamma'(2^{k-1}t)}{\gamma'(2^{k-2}t)}\cdots\frac{\gamma'(2t)}{\gamma'(t)}$ is decreasing on $(0,\infty)$ for any $k\in \mathbb{N}$. By Remark \ref{remark 1.1}, we know that $1\leq\frac{t\gamma'(t)}{\gamma(t)}\leq C_4$ for any $t\in (0,\infty)$. Noticing $\gamma$ is either odd or even, $\gamma'$ is increasing on $(0,\infty)$, \eqref{eq:3.1}, $|t|\leq 2$, $|ht-s|\leq 2$ and $\frac{\gamma'(2t)}{\gamma'(t)}\leq C_1$ for any $t\in (0,\infty)$, we obtain
\begin{align}\label{eq:3.12}
|Q'(t)|
   \geq&  \left|u(x)2^{n(x)+k}\gamma' (2^{n(x)+k} t)\right|-\left|u(z)2^{n(z)+k} \gamma '(2^{n(z)+k}[ ht-s ]  )\right|h\\
    \geq&  \left|\frac{1}{\gamma(2^{n(x)+1})}2^{n(x)+k}\gamma' \left(2^{n(x)+k} \frac{1}{2}\right)\right|-\left|\frac{1}{\gamma(2^{n(z)})}2^{n(z)+k} \gamma '(2^{n(z)+k}2  )\right|h\nonumber\\
    =&  \left|\frac{2^{n(x)+1}\gamma'(2^{n(x)+1})}{\gamma(2^{n(x)+1})}\frac{2^{n(x)+k}}{2^{n(x)+1}}\frac{\gamma' (2^{n(x)+k} \frac{1}{2})}{\gamma'(2^{n(x)+k})}\frac{\gamma' (2^{n(x)+k})}{\gamma'(2^{n(x)})}\frac{\gamma'(2^{n(x)})}{\gamma'(2^{n(x)+1})}\right|\nonumber\\
    &  -\left|\frac{2^{n(z)}\gamma'(2^{n(z)})}{\gamma(2^{n(z)})}\frac{2^{n(z)+k}}{2^{n(z)}}\frac{\gamma' (2^{n(z)+k} 2)}{\gamma'(2^{n(x)+k})}\frac{\gamma' (2^{n(z)+k})}{\gamma'(2^{n(z)})}\right|h\nonumber\\
    \geq&  \frac{1}{2C_1^2} 2^k \frac{\gamma' (2^{n(x)+k})}{\gamma'(2^{n(x)})}-C_1C_4 2^k \frac{\gamma' (2^{n(x)+k})}{\gamma'(2^{n(x)})} h
    \geq  \left( \frac{1}{4C_1^2}  \right)2^k \frac{\gamma' (2^{n(x)+k})}{\gamma'(2^{n(x)})}. \nonumber
\end{align}
As \eqref{eq:3.12} and using the fact that $\frac{t\gamma''(t)}{\gamma'(t)}\leq C_2$ for any $t\in (0,\infty)$ and $h\leq 1$, we find that
\begin{align}\label{eq:3.13}
|Q''(t)|
   \leq&  \left|u(x)2^{2(n(x)+k)}\gamma'' (2^{n(x)+k} t)\right|+\left|u(z)2^{2(n(z)+k)} \gamma''(2^{n(z)+k}[ ht-s ]  )h^2\right|\\
      =&  \left|u(x)2^{2(n(x)+k)}\frac{(2^{n(x)+k} t)\gamma'' (2^{n(x)+k} t)}{\gamma' (2^{n(x)+k} t)} \frac{\gamma' (2^{n(x)+k} t)}{(2^{n(x)+k} t)} \right|\nonumber\\
      &+\left|u(z)2^{2(n(z)+k)} \frac{(2^{n(z)+k}[ ht-s ])\gamma''(2^{n(z)+k}[ ht-s ])}{\gamma'(2^{n(z)+k}[ht-s]) }\frac{\gamma'(2^{n(z)+k}[ ht-s ])}{(2^{n(z)+k}[ ht-s ])}h^2\right|\nonumber\\
      \leq& 2C_2 \left|u(x)2^{(n(x)+k)} \gamma' (2^{n(x)+k} 2) \right|+2C_2\left|u(z)2^{(n(z)+k)} \gamma'(2^{n(z)+k}2)\right|\nonumber\\
      \leq& 2C_1C_2 \left|\frac{1}{\gamma(2^{n(x)})}2^{(n(x)+k)} \gamma' (2^{n(x)+k} ) \right|+2C_1C_2\left|\frac{1}{\gamma(2^{n(z)})}2^{(n(z)+k)} \gamma'(2^{n(z)+k})\right|\nonumber\\
     =& 2C_1C_2 \left|\frac{2^{n(x)}\gamma'(2^{n(x)})}{\gamma(2^{n(x)})}\frac{2^{(n(x)+k)} }{2^{n(x)}}\frac{\gamma' (2^{n(x)+k} )}{\gamma'(2^{n(x)})} \right|\nonumber\\
     &+2C_1C_2\left|\frac{2^{n(z)}\gamma'(2^{n(z)})}{\gamma(2^{n(z)})}\frac{2^{(n(z)+k)}}{2^{n(z)}} \frac{\gamma'(2^{n(z)+k})}{\gamma'(2^{n(z)})}\right|\nonumber\\
 \leq& 2C_1C_2C_4 2^k\frac{\gamma' (2^{n(x)+k})}{\gamma'(2^{n(x)})}+  2C_1C_2C_4 2^k\frac{\gamma' (2^{n(z)+k})}{\gamma'(2^{n(z)})}
 \leq 4C_1C_2C_4 2^k\frac{\gamma' (2^{n(x)+k})}{\gamma'(2^{n(x)})}. \nonumber
\end{align}
Combining \eqref{eq:3.12} and \eqref{eq:3.13}, using Lemma \ref{lemma 2.1} and (\cite{S}, P.334, Corollary), and the fact that $\gamma'$ is increasing on $(0,\infty)$, we conclude that
\begin{align}\label{eq:3.14}
&\left|\mathrm{p.\,v.}\int_{-\infty}^{\infty}  e^{iu(x)\gamma (2^{n(x)+k} t)-iu(z)\gamma \left(2^{n(z)+k}[ht-s]  \right)}  \frac{\psi\left( ht-s\right)}{ht-s}   \frac{\psi(t)}{t}  \,\textrm{d}t\right|\\
&\quad \lesssim \frac{1}{\left( \frac{1}{4C_1^2}  \right)2^k \frac{\gamma' (2^{n(x)+k})}{\gamma'(2^{n(x)})}} +\frac{4C_1C_2C_4 2^k\frac{\gamma' (2^{n(x)+k})}{\gamma'(2^{n(x)})}}{\left[ \left( \frac{1}{4C_1^2}  \right)2^k \frac{\gamma' (2^{n(x)+k})}{\gamma'(2^{n(x)})}  \right] ^2 }
 \lesssim   \frac{1}{2^k}. \nonumber
\end{align}
Thus, in this case, \eqref{eq:3.8} holds with $r_2=1$ and arbitrary positive constant $r_1$.

{\bf Case~B}  $\frac{1}{4 C_1^3C_4}<h\leq 1$.

If $|s|\leq 2^{-\frac{k}{8}}$, since $\psi:\ \mathbb{R}\rightarrow\mathbb{R}$ is supported on $\left\{t\in \mathbb{R}:\ \frac{1}{2}\leq |t|\leq 2\right\}$, it follows that the integral in \eqref{eq:3.8} is
bounded by $C$. Thus, in this case, \eqref{eq:3.8} holds with $r_1=\frac{1}{8}$ and arbitrary positive constant $r_2$. In the remainder, we only to consider the case $|s|\geq 2^{-\frac{k}{8}}$. We write
\begin{equation}\label{eq:3.15}
\left(
\begin{array}{ccc}
 Q'(t)    \\
 Q''(t)
\end{array}
\right)
=
M_{t,s}
\Upsilon,
\end{equation}
where $M_{t,s}$ is the $2\times2$ matrix
\begin{equation}\label{eq:3.16}
M_{t,s}:=
\left(
\begin{array}{ccc}
 1&  h  \\
 \frac{2^{n(x)+k}\gamma''(2^{n(x)+k} t)}{ \gamma'(2^{n(x)+k}t)}&\frac{2^{n(z)+k}\gamma''(2^{n(z)+k} [ht-s])}{\gamma'(2^{n(z)+k}[ht-s])}h^2
\end{array}
\right)
\end{equation}
and $\Upsilon$ is the vector
\begin{equation}\label{eq:3.17}
\Upsilon:=
\left(
\begin{array}{ccc}
  u(x)2^{n(x)+k}\gamma' (2^{n(x)+k} t)    \\
 -u(z)2^{n(z)+k} \gamma '(2^{n(z)+k}[ht-s]  )
\end{array}
\right).
\end{equation}
We may compute immediately as \eqref{eq:3.12} that
\begin{align}\label{eq:3.18}
|\Upsilon|
   \geq  \left|u(x)2^{n(x)+k}\gamma' (2^{n(x)+k} t)\right|
   \geq   \frac{1}{2C_1^2} 2^k \frac{\gamma' (2^{n(x)+k})}{\gamma'(2^{n(x)})}.
\end{align}
Moreover, let
$$a_0:=\frac{2^{n(x)+k}t\gamma''(2^{n(x)+k} t)}{ \gamma'(2^{n(x)+k}t)}$$
and
$$b_0:=\frac{2^{n(z)+k}( ht-s )\gamma''(2^{n(z)+k} [ht-s])}{\gamma'(2^{n(z)+k}[ht-s])}.$$
We can rewrite $M_{t,s}$ as
\begin{equation}\label{eq:3.19}
M_{t,s}=
\left(
\begin{array}{ccc}
 1&  h  \\
 a_0 \frac{1}{t} & b_0 \frac{h^2}{ht-s}
\end{array}
\right).
\end{equation}
From the fact that $\frac{t\gamma''(t)}{\gamma'(t)}\leq C_2$ for any $t\in (0,\infty)$, it implies that $|a_0 |\leq C_2$ and $| b_0| \leq C_2$, which further follows that
\begin{align}\label{eq:3.21}
\|M_{t,s}\|=\sup_{|x|=1}|M_{t,s}x|\lesssim 1.
\end{align}
From \eqref{eq:3.16} and Theorem \ref{theorem 1.1}(iv), together with the fact that $|s|\leq 4$, $h=\frac{2^{n(x)}}{2^{n(z)}}$ and the generalised mean value theorem, we have that there exists a positive constant $\theta \in [0, 1]$ such that
\begin{align}\label{eq:3.23}
|\textrm{det} M_{t,s}|
   =& h2^{n(x)+k} \left|\frac{\gamma''(2^{n(x)+k} t-2^{n(z)+k}  s )}{\gamma'(2^{n(x)+k} t-2^{n(z)+k}  s)}- \frac{\gamma''(2^{n(x)+k} t)}{\gamma'(2^{n(x)+k}t)}  \right|\\
 =& h2^{n(x)+k} \left|\left(\frac{\gamma''}{\gamma'}\right)'\left(2^{n(x)+k} t-2^{n(z)+k}  s\theta\right) 2^{n(z)+k}  s\right|\nonumber\\
  \geq& C_3 h2^{n(x)+k}  \frac{1}{\left[2^{n(x)+k} t-2^{n(z)+k}  s\theta\right]^2}  \left| 2^{n(x)+k}  s\right|\nonumber\\
   =& C_3 h  \frac{1}{\left( t-\frac{1}{h} s\theta\right)^2}  \left|  s\right|
   \gtrsim 2^{-\frac{k}{8}}.\nonumber
\end{align}
Combining \eqref{eq:3.18}, \eqref{eq:3.21}, \eqref{eq:3.23}, and Lemma \ref{lemma 2.2} with $n=2$, we have therefore
\begin{align}\label{eq:3.24}
  M_{t,s}\Upsilon \geq |\textrm{det} M_{t,s}| \|M_{t,s}\|^{-1} |\Upsilon|  \gtrsim 2^{\frac{7k}{8}} \frac{\gamma' (2^{n(x)+k})}{\gamma'(2^{n(x)})}
\end{align}
and so
\begin{align}\label{eq:3.25}
  \sqrt{ \left[Q'(t)\right]^2+\left[Q''(t)\right]^2 } \gtrsim 2^{\frac{7k}{8}}\frac{\gamma' (2^{n(x)+k})}{\gamma'(2^{n(x)})} .
\end{align}
By pigeonholing, there are two cases: If $|Q'(t)|\gtrsim2^{\frac{7k}{8}} \frac{\gamma' (2^{n(x)+k})}{\gamma'(2^{n(x)})}$, notice that $h=\frac{2^{n(x)}}{2^{n(z)}}$, by Lemma \ref{lemma 2.3}, let $a:=u(x)2^{n(x)+k}$, $b:=u(z)2^{n(z)+k} h$, $c=2^{n(z)+k}s$, $d:=2^{\frac{7k}{8}} \frac{\gamma' (2^{n(x)+k})}{\gamma'(2^{n(x)})}$ and $t:=2^{n(x)+k} t$, we see that this case only happen on at most a finite number of intervals, and the number of intervals is independent of $x,z,s,k$ and $u$. Using \eqref{eq:3.13}, from Lemma \ref{lemma 2.1} and (\cite{S}, P.334, Corollary), similarly to \eqref{eq:3.14} we obtain that the integral in \eqref{eq:3.8} on this portion is established with $r_2=\frac{3}{4}$ and arbitrary positive constant $r_1$. If $|Q''(t)|\gtrsim 2^{\frac{7k}{8}} \frac{\gamma' (2^{n(x)+k})}{\gamma'(2^{n(x)})}$, by our argument in the first case, this case also only happen on at most a finite number of intervals, by the van der Corput lemma's, similarly to \eqref{eq:3.14} we conclude that the integral in \eqref{eq:3.8} on this portion is established with $r_2=\frac{7}{16}$ and arbitrary positive constant $r_1$. Altogether we have now show that the integral in \eqref{eq:3.8} is established with $r_2=\frac{7}{16}$ and arbitrary positive constant $r_1$. This finishes the proof of Proposition \ref{proposition 3.1}.
\end{proof}

\section{ Proof of Theorem \ref{theorem 1.1}}

First we notice that the $L^2(\mathbb{R}^{2})$ estimate for \eqref{Hilbert transform} follows from Theorem \ref{theorem 1.2}. In deed, from \cite{PS}, it follows that
$$\|H_{u,\gamma}\|_{L^2(\mathbb{R}^{2})\rightarrow L^2(\mathbb{R}^{2})} \leq \sup_{\lambda\in \mathbb{R}} \|S_{\lambda}\|_{L^2(\mathbb{R})\rightarrow L^2(\mathbb{R})} ,$$
where
$$S_\lambda f(x):=\mathrm{p.\,v.}\int_{-\infty}^{\infty}  e^{-i\lambda u(x)\gamma(t)}  f(x-t)\,\frac{\textrm{d}t}{t}, \quad\forall\, x\in\mathbb{R}.$$
Since the $L^2(\mathbb{R}^{2})$ boundedness of $H_{u,\gamma}$ will not depend on $u$, we need only to establish the $L^2(\mathbb{R})$ estimate for
$$\mathcal{C}_{u,\gamma}f(x)= \mathrm{p.\,v.}\int_{-\infty}^{\infty}  e^{iu(x)\gamma (t)}  f(x-t)\,\frac{\textrm{d}t}{t}, \quad\forall\, x\in\mathbb{R},$$
with the bound independent of $u$. This has been proved in Theorem \ref{theorem 1.2}.

\subsection{Single annulus $L^p(\mathbb{R}^{2})$ estimate for the Hilbert transform $H_{u,\gamma}$}\label{section 3}

Before establishing the $L^p(\mathbb{R}^{2})$ estimate for \eqref{Hilbert transform}, we warm up ourselves by establishing the following single annulus $L^p(\mathbb{R}^{2})$ estimate. There are many other works about this topic, such as \cite{B1} and \cite{LL1}.
Recall that $\psi:\ \mathbb{R}\rightarrow\mathbb{R}$ is supported on $\left\{t\in \mathbb{R}:\ \frac{1}{2}\leq |t|\leq 2\right\}$ such that $0\leq \psi(t)\leq 1$ and $\Sigma_{l\in \mathbb{Z}} \psi_l(t)=1$ for any $t\neq 0$, where $\psi_l(t)=\psi (2^{-l}t)$. For any $l\in \mathbb{Z}$, let $P_l$ denotes the Littlewood-Paley projection in the second variable corresponding to $\psi_l$. That is
$$P_lf(x_1,x_2):=\int_{-\infty}^{\infty}  f(x_1,x_2-z)\check{\psi}_l(z)\,\textrm{d}z.$$

\begin{theorem}\label{theorem 1.3}
Let $u$ and $\gamma$ be the same as in Theorem \ref{theorem 1.1}. Then for any given $p\in (1,\infty)$, we have
$$\|H_{u,\gamma}P_lf\|_{L^{p}(\mathbb{R}^{2})}\leq C \|P_lf\|_{L^{p}(\mathbb{R}^{2})},$$
uniformly in $l\in \mathbb{Z}$, and the bound $C$ is a positive constant independent of $u$.
\end{theorem}

\begin{proof}[Proof of Theorem \ref{theorem 1.3}]
By an anisotropic scaling
$$x_1\rightarrow x_1, x_2\rightarrow 2^{-l}x_2,$$
we consider only the case that $l=0$. Let us set
$$H_{u,\gamma,k}P_0f(x_1,x_2):=\int_{-\infty}^{\infty} P_0f(x_1-t,x_2-u(x_1)\gamma(t))\psi_k(t)\,\frac{\textrm{d}t}{t}.$$
Let $n:\ \mathbb{R}\rightarrow\mathbb{Z}$ such that for any $x_1\in \mathbb{R}$
\begin{align}\label{eq:3.27}
\frac{1}{\gamma (2^{n(x_1)+1})}\leq |u(x_1)|\leq \frac{1}{\gamma (2^{n(x_1)})} .
\end{align}
We decompose
\begin{align}\label{eq:yly2}
H_{u,\gamma}P_0f(x_1,x_2)
   =&  \sum_{k\leq n(x_1)-1}  H_{u,\gamma,k}P_0f(x_1,x_2)+  \sum_{k\geq n(x_1)} H_{u,\gamma,k}P_0f(x_1,x_2)\\
   =:& H^{(1)}_{u,\gamma}P_0f(x_1,x_2)+H^{(2)}_{u,\gamma}P_0f(x_1,x_2).\nonumber
\end{align}

For $H^{(1)}_{u,\gamma}P_0f$, let $\rho$ be a non-negative smooth function supported on $\left\{\xi\in \mathbb{R}:\ \frac{1}{4}\leq |\xi|\leq 4\right\}$ such that $\rho=1$ on $\left\{\xi\in \mathbb{R}:\ \frac{1}{2}\leq |\xi|\leq 2\right\}$, and let $\mathbb{P}_0f(x_1,x_2):=\int_{-\infty}^{\infty}  f(x_1,x_2-s)\check{\rho}(s)\,\textrm{d}s$.
By Fourier transform, it is easy to check that
\begin{align}\label{eq:717}\mathbb{P}_0P_0f=P_0f.\end{align}
We first consider $H^{(1)}_{u,\gamma}\mathbb{P}_0f$. Let $\sum_{k\leq n(x_1)-1} \psi_k(t)=:\phi (t)$, then
\begin{align}\label{eq:3.28}
H^{(1)}_{u,\gamma}\mathbb{P}_0f(x_1,x_2)
   =  \mathrm{p.\,v.} \int_{|t|\leq   2^{n(x_1)}}\mathbb{P}_0f(x_1-t,x_2-u(x_1)\gamma(t))\phi (t)\,\frac{\textrm{d}t}{t}.
\end{align}
Let us consider an approximate operator
$$\tilde{H}\mathbb{P}_0f(x_1,x_2):= \mathrm{p.\,v.}\int_{|t|\leq   2^{n(x_1)}}\mathbb{P}_0f(x_1-t,x_2)\phi (t)\,\frac{\textrm{d}t}{t}.$$
As \eqref{eq:bnu2}, we have
\begin{align}\label{eq:3.29}
\tilde{H}\mathbb{P}_0f(x_1,x_2)
  \lesssim M_1\mathbb{P}_0f(x_1,x_2)+\mathcal{\tilde{H}}^*_1\mathbb{P}_0f(x_1,x_2).
\end{align}
Here and hereafter, $\mathcal{\tilde{H}}^*_1$ denotes the maximal truncated Hilbert transform applied in the first variable, $M_1$ and $M_2$ denote the Hardy-Littlewood maximal operator applied in the first variable and the second variable, respectively. Since both $M_1$ and $\mathcal{\tilde{H}}^*_1$ are known to be bounded on $L^p(\mathbb{R}^2)$, from \eqref{eq:3.29} we may conclude that
\begin{align}\label{eq:3.30}
\|\tilde{H}\mathbb{P}_0f\|_{L^{p}(\mathbb{R}^2)}\lesssim \|\mathbb{P}_0f\|_{L^{p}(\mathbb{R}^2)}\lesssim \|f\|_{L^{p}(\mathbb{R}^2)}
\end{align}
for any given $p\in (1,\infty)$.

Now we turn to the difference between $H^{(1)}_{u,\gamma}\mathbb{P}_0f$ and $\tilde{H}\mathbb{P}_0f$, which can be written as
\begin{align}\label{eq:3.31}
\mathrm{p.\,v.}\int_{|t|\leq   2^{n(x_1)}}\int_{-\infty}^{\infty}  f(x_1-t,x_2-z) \left[ \check{\rho}(z-u(x_1)\gamma(t))- \check{\rho}(z) \right]\textrm{d}z\phi (t)\,\frac{\textrm{d}t}{t}.
\end{align}
Since $\gamma$ is increasing on $(0,\infty)$ and $|t|\leq   2^{n(x_1)}$, we have $|u(x_1)\gamma(t)|\leq |u(x_1)|\gamma(2^{n(x_1)})\leq 1$. Then apply the mean value theorem to obtain
$$| \check{\rho}(z-u(x_1)\gamma(t))- \check{\rho}(z)|\lesssim \sum_{m\in \mathbb{Z}}\frac{1}{(|m-1|+1)^2}\chi_{[m,m+1]}(z)|u(x_1)\gamma(t)|. $$
Due to the fact that $\sum_{m\in \mathbb{Z}}\frac{1}{(|m-1|+1)^2}\lesssim 1$, it suffices to bound the operator defined by setting, for any fixed $m\in \mathbb {Z}$,
\begin{align}\label{eq:3.32}
K_mf(x_1,x_2):=\int_m^{m+1}\int_{|t|\leq   2^{n(x_1)}}|f(x_1-t,x_2-z)|\frac{|u(x_1)\gamma(t)|}{|t|} \phi (t)\,\textrm{d}t\textrm{d}z
\end{align}
with a  bound independent of $m$ and $u$. By Minkowski's inequality, \eqref{eq:3.27} and noticing that $\frac{\gamma(t)}{t}$ is increasing on $(0,\infty)$, we have
\begin{align}\label{eq:3.33}
\|K_mf(\cdot_1,\cdot_2)\|^p_{L^{p}(\mathbb{R}^{2})}
&\leq  \int_{-\infty}^{\infty} \left\{ \int_m^{m+1} \int_{|t|\leq   2^{n(x_1)}}  \|f(x_1-t,\cdot_2)\|_{L^{p}(\mathbb{R}^1_{x_2})}    \frac{|u(x_1)\gamma(t)|}{|t|} \phi (t) \,\textrm{d}t\,\textrm{d}z \right\}^p\,\textrm{d}x_1 \\
&\leq \int_{-\infty}^{\infty}\left\{\int_{|t|\leq   2^{n(x_1)}}  \|f(x_1-t,\cdot_2)\|_{L^{p}(\mathbb{R}^1_{x_2})}    \frac{|u(x_1)\gamma( 2^{n(x_1)})|}{| 2^{n(x_1)}|} \phi (t) \,\textrm{d}t \right\}^p\,\textrm{d}x_1 \nonumber\\
&\leq \int_{-\infty}^{\infty}\left\{ \frac{1}{2^{n(x_1)}}\int_{|t|\leq   2^{n(x_1)}}  \|f(x_1-t,\cdot_2)\|_{L^{p}(\mathbb{R}^1_{x_2})}  \, \textrm{d}t \right\}^p\,\textrm{d}x_1 \nonumber\\
&\lesssim  \int_{-\infty}^{\infty}\left[ M (\|f(\cdot,\cdot_2)\|_{L^{p}(\mathbb{R}^1_{x_2})})(x_1) \right]^p\,\textrm{d}x_1 \lesssim  \|f\|^p_{L^{p}(\mathbb{R}^{2})}, \nonumber
\end{align}
where $p\in(1,\infty)$. From \eqref{eq:3.30} and \eqref{eq:3.33}, it follows that
$$\|H^{(1)}_{u,\gamma}\mathbb{P}_0f\|_{L^{p}(\mathbb{R}^2)}\lesssim\|f\|_{L^{p}(\mathbb{R}^2)}.$$
Therefore, by \eqref{eq:717}, it implies
$$\|H^{(1)}_{u,\gamma}P_0f\|_{L^{p}(\mathbb{R}^2)}\lesssim \|P_0f\|_{L^{p}(\mathbb{R}^2)}.$$

For $H^{(2)}_{u,\gamma}P_0f$, let $f:=P_0f$, we can write
\begin{eqnarray*}
H^{(2)}_{u,\gamma}f(x_1,x_2)
     = \sum_{k\geq 0} \int_{-\infty}^{\infty} f(x_1-t,x_2-u(x_1)\gamma(t))\psi_{n(x_1)+k}(t)\,\frac{\textrm{d}t}{t}.
\end{eqnarray*}
By Minkowski's inequality, similarly to \eqref{eq:3.33}, for any given $p\in (1,\infty)$, we have
\begin{align}\label{eq:3.34}
\left\| \int_{-\infty}^{\infty} f(\cdot_1-t,\cdot_2-u(\cdot_1)\gamma(t))\psi_{n(\cdot_1)+k}(t)\,\frac{\textrm{d}t}{t}\right\|_{L^{p}(\mathbb{R}^{2})}
\lesssim \|f\|_{L^{p}(\mathbb{R}^{2})}.
\end{align}
From \eqref{eq:3.3} we already have
$$\left\| \int_{-\infty}^{\infty} e^{iu(\cdot)\gamma (t)}  f(\cdot-t)   \psi_{k+n(\cdot)}(t) \, \frac{\textrm{d}t}{t} \right\|_{L^{2}(\mathbb{R})}\lesssim 2^{-\omega_0 k}\|f\|_{L^{2}(\mathbb{R})}.$$
Therefore,
\begin{align}\label{eq:y}
\left\| \int_{-\infty}^{\infty} f(\cdot_1-t,\cdot_2-u(\cdot_1)\gamma(t))\psi_{n(\cdot_1)+k}(t)\,\frac{\textrm{d}t}{t}\right\|_{L^{2}(\mathbb{R}^{2})}\lesssim 2^{-\omega_0 k}\|f\|_{L^{2}(\mathbb{R}^2)}.
\end{align}
By interpolation between \eqref{eq:3.34} and \eqref{eq:y}, sum over $k\geq 0$, which leads to
$$\|H^{(2)}_{u,\gamma}f\|_{L^{p}(\mathbb{R}^2)}\lesssim \|f\|_{L^{p}(\mathbb{R}^2)}$$
for any given $p\in (1,\infty)$. This finishes the proof of Theorem \ref{theorem 1.3}.
\end{proof}

\subsection{$L^p(\mathbb{R}^{2})$ estimate for the Hilbert transform $H_{u,\gamma}$}\label{section 4}

Now we turn to the $L^p(\mathbb{R}^{2})$ estimate for the Hilbert transform $H_{u,\gamma}$ defined in \eqref{Hilbert transform} for any given $p\in (1,\infty)$. Our proof rely crucially on the commutation relation between $H_{u,\gamma}$ and $P_l$, then we can reduce our attention to a square function. As before, we also decompose our operator into a low frequency part and a high frequency part, the low frequency part is controlled by the Hardy-Littlewood maximal operator and the maximal truncated Hilbert transform. For the high frequency part, which also is repressed as a series of operators. Building on the already proofed $L^2(\mathbb{R}^{2})$ estimate with bound $2^{-\omega_0 k}$ and the strategy of interpolation, it suffices to obtained a $L^p(\mathbb{R}^{2})$ estimate with bound $k^2$. This unusual $L^p(\mathbb{R}^{2})$ bound can be achieved by the shifted maximal operator, which form a pointwise estimate for taking average along variable plane curve $u(x_1)\gamma$.

\begin{proof}[Proof of Theorem \ref{theorem 1.1}]

We note that the commutation relation
$$H_{u,\gamma}P_l=P_l H_{u,\gamma}$$
holds for any $l\in \mathbb{Z}$. By Littlewood-Paley theory it is enough to show that
\begin{align}\label{eq:3.36}
\left\|\left[\sum_{l\in \mathbb{Z}}\left|H_{u,\gamma}P_lf\right|^2\right]^{\frac{1}{2}}\right\|_{L^{p}(\mathbb{R}^{2})}\lesssim \|f\|_{L^{p}(\mathbb{R}^{2})}.
\end{align}
As \eqref{eq:3.27}, for any $l\in \mathbb{Z}$, define $n_l:\ \mathbb{R}\rightarrow\mathbb{Z}$ such that for any $x_1\in \mathbb{R}$
\begin{align}\label{eq:y3}
\frac{1}{\gamma (2^{n_l(x_1)+1})}\leq 2^l|u(x_1)|\leq \frac{1}{\gamma (2^{n_l(x_1)})} .
\end{align}
Similarly to \eqref{eq:yly2}, we decompose $H_{u,\gamma}P_l$ as
\begin{align}\label{eq:3.37}
H_{u,\gamma}P_l f(x_1,x_2)
=& \sum_{k\leq n_l(x_1)-1}  \int_{-\infty}^{\infty} P_lf(x_1-t,x_2-u(x_1)\gamma(t)) \psi_{k}(t)\,\frac{\textrm{d}t}{t}\\
& + \sum_{k\geq 0}  \int_{-\infty}^{\infty} P_lf(x_1-t,x_2-u(x_1)\gamma(t)) \psi_{k+n_l(x_1)}(t)\,\frac{\textrm{d}t}{t} \nonumber\\
=:& H^{(I)}_{u,\gamma}P_lf(x_1,x_2)+\sum_{k\geq 0}H_{u,\gamma,k+n_l(x_1)}P_lf(x_1,x_2).  \nonumber
\end{align}
Using the triangle inequality,  the left term of \eqref{eq:3.36} can be controlled by
\begin{equation}\label{eq:3.38}
\left\|\left[\sum_{l\in \mathbb{Z}}\left|H^{(I)}_{u,\gamma}P_lf\right|^2\right]^{\frac{1}{2}}\right\|_{L^{p}(\mathbb{R}^{2})}+ \sum_{k\geq 0}\left\|\left[\sum_{l\in \mathbb{Z}}\left|H_{u,\gamma,k+n_l(\cdot_1)}P_lf\right|^2\right]^{\frac{1}{2}}\right\|_{L^{p}(\mathbb{R}^{2})}.
\end{equation}

As before, for the low frequency part in \eqref{eq:3.38}, let $\sum_{k\leq n_l(x_1)-1} \psi_k(t)=:\phi (t)$,
$$ H^{(I)}_{u,\gamma}P_lf(x_1,x_2)= \mathrm{p.\,v.}\int_{|t|\leq   2^{n_l(x_1)}} P_lf(x_1-t,x_2-u(x_1)\gamma(t)) \phi(t)\,\frac{\textrm{d}t}{t} .$$
Let
$$\tilde{H}f(x_1,x_2):= \mathrm{p.\,v.}\int_{|t|\leq   2^{n_l(x_1)}}f(x_1-t,x_2)\phi (t)\,\frac{\textrm{d}t}{t}.$$
As \eqref{eq:bnu2}, we may obtain
\begin{align}\label{eq:3.39}
\tilde{H}P_lf(x_1,x_2) \lesssim M_1P_lf(x_1,x_2)+\mathcal{\tilde{H}}^*_1P_lf(x_1,x_2),
\end{align}
The vector-valued estimate for $M_1$ follows from the corresponding estimate for the one dimensional Hardy-Littlewood maximal function. Similarly, the vector-valued estimate for $\mathcal{\tilde{H}}^*_1$ follows from Cotlar's inequality and the vector-valued estimate for the Hilbert transform and the maximal function. Then from \eqref{eq:3.39} and the Littlewood-Paley theory one may obtain
\begin{align}\label{eq:3.40}
\left\|\left[\sum_{l\in \mathbb{Z}}\left|\tilde{H}P_lf\right|^2\right]^{\frac{1}{2}}\right\|_{L^{p}(\mathbb{R}^{2})}\lesssim \left\|\left[\sum_{l\in \mathbb{Z}}\left|P_lf\right|^2\right]^{\frac{1}{2}}\right\|_{L^{p}(\mathbb{R}^{2})}\lesssim\|f\|_{L^{p}(\mathbb{R}^{2})}.
\end{align}

Now we turn to the difference between $H^{(I)}_{u,\gamma}P_lf$ and $\tilde{H}P_lf$. Recall that $\rho$ is a non-negative smooth function supported on $\left\{s\in \mathbb{R}:\,\frac{1}{4}\leq |\xi|\leq 4\right\}$ and equals to $1$ on $\left\{s\in \mathbb{R}:\,\frac{1}{2}\leq |s|\leq 2\right\}$, let $\rho_l(s):=\rho (2^{-l}s)$, $l\in \mathbb{Z}$, and $\mathbb{P}_lf(x_1,x_2):=\int_{-\infty}^{\infty} f(x_1,x_2-s)\check{\rho_l}(s)\,\textrm{d}s$.
By Fourier transform, it is easy to see that
\begin{align}\label{eq:y0}\mathbb{P}_lP_lf=P_lf.\end{align}
The difference between $H^{(I)}_{u,\gamma}\mathbb{P}_lf$ and $\tilde{H}\mathbb{P}_lf$ can be written as
\begin{align}\label{eq:3.41}
\mathrm{p.\,v.}\int_{|t|\leq   2^{n_l(x_1)}} \int_{-\infty}^{\infty} f(x_1-t,x_2-s) \left[ \check{\rho}_l(s-u(x_1)\gamma(t))- \check{\rho}_l(s) \right]\,\textrm{d}s\phi (t)\,\frac{\textrm{d}t}{t}.
\end{align}
By the mean value theorem we have
\begin{align}\label{eq:3.42}
| \check{\rho_l}(s-w)- \check{\rho_l}(s) |\lesssim |w|2^{2l}2^{-2j}
\end{align}
if $|w|\leq 2^{-l}$ and $s$ is in the annulus $2^{-l+j-1}\leq |s|\leq 2^{-l+j}$ for $j\in \mathbb{N}$. For $j=0$ the estimate holds for all $|s|\leq 2^{-l}$. Since $\gamma$ is increasing on $(0,\infty)$ and $\gamma$ is either odd or even, from \eqref{eq:y3}, it implies that $2^l|u(x_1)\gamma(t)|\leq 2^l|u(x_1)|\gamma(2^{n_l(x_1)})\leq 1$ for all $|t|\leq   2^{n_l(x_1)}$. Thus the absolute value of \eqref{eq:3.41} can be estimated by a positive constant times
\begin{align}\label{eq:y4}
\sum_{j\in \mathbb{N}} \int_{|t|\leq   2^{n_l(x_1)}}\int_{|s|\leq 2^{-l+j}} |f(x_1-t,x_2-s)|  2^{2l}2^{-2j}|u(x_1)|\left|\frac{\gamma(t)}{t}\right| \,\textrm{d}s\,\textrm{d}t.
\end{align}
Noticing that $\frac{\gamma(t)}{t}$ is increasing on $(0,\infty)$, and $\gamma$ is either odd or even and \eqref{eq:y3}, we can bound \eqref{eq:y4}  by
\begin{align}\label{eq:y5}
&\sum_{j\in \mathbb{N}} \int_{|t|\leq   2^{n_l(x_1)}}\int_{|s|\leq 2^{-l+j}} |f(x_1-t,x_2-s)|  2^{2l}2^{-2j}|u(x_1)|\left|\frac{\gamma( 2^{n_l(x_1)})}{ 2^{n_l(x_1)}}\right| \,\textrm{d}s\,\textrm{d}t\\
&\quad \lesssim\sum_{j\in \mathbb{N}}2^{-j} \frac{1}{2^{n_l(x_1)}} \int_{|t|\leq   2^{n_l(x_1)}} \frac{1}{2^{-l+j}} \int_{|s|\leq 2^{-l+j}} |f(x_1-t,x_2-s)|  \,\textrm{d}s\,\textrm{d}t \lesssim M_1M_2f(x_1,x_2). \nonumber
\end{align}
Therefore, from the vector-valued estimate for $M_1,M_2$, the Littlewood-Paley theory and \eqref{eq:y0}, and the triangle inequality yield
\begin{align}\label{eq:3.43}
\left\|\left[\sum_{l\in \mathbb{Z}}\left|H^{(I)}_{u,\gamma}P_lf-\tilde{H}P_lf\right|^2\right]^{\frac{1}{2}}\right\|_{L^{p}(\mathbb{R}^{2})}
&\quad =\left\|\left[\sum_{l\in \mathbb{Z}}\left|H^{(I)}_{u,\gamma}\mathbb{P}_lP_lf-\tilde{H}\mathbb{P}_lP_lf\right|^2\right]^{\frac{1}{2}}\right\|_{L^{p}(\mathbb{R}^{2})}\\
&\quad \lesssim\left\|\left[\sum_{l\in \mathbb{Z}}\left|M_1M_2P_lf\right|^2\right]^{\frac{1}{2}}\right\|_{L^{p}(\mathbb{R}^{2})}
\lesssim  \left\|f\right\|_{L^{p}(\mathbb{R}^{2})}\nonumber.
\end{align}
From \eqref{eq:3.40} and \eqref{eq:3.43}, it follows that
\begin{align}\label{eq:3.44}
\left\|\left[\sum_{l\in \mathbb{Z}}\left|H^{(I)}_{u,\gamma}P_lf\right|^2\right]^{\frac{1}{2}}\right\|_{L^{p}(\mathbb{R}^{2})}\lesssim \left\|f\right\|_{L^{p}(\mathbb{R}^{2})}.
\end{align}

For the high frequency part in \eqref{eq:3.38}, It is enough to show that there exists a convergent series $\left\{C_k\right\}_{k=0}^{+\infty}$ such that for any $k\geq 0$,
\begin{align}\label{eq:3.45}
\left\|\left[\sum_{l\in \mathbb{Z}}\left|H_{u,\gamma,k+n_l(\cdot_1)}P_lf\right|^2\right]^{\frac{1}{2}}\right\|_{L^{p}(\mathbb{R}^{2})}\lesssim  C_k\left\|f\right\|_{L^{p}(\mathbb{R}^{2})}.
\end{align}

First for $p=2$, noticing that the bound in \eqref{eq:3.3} is independent of $u$, therefore, we can replace $u$ by $2^lu$ in \eqref{eq:3.3}. By Littlewood-Paley theory, as \eqref{eq:3.3} we have
\begin{align}\label{eq:3.46}
\left\|\left[\sum_{l\in \mathbb{Z}}\left|H_{u,\gamma,k+n_l(\cdot_1)}P_lf\right|^2\right]^{\frac{1}{2}}\right\|_{L^{2}(\mathbb{R}^{2})}\lesssim 2^{-\omega_0 k} \left\|f\right\|_{L^{2}(\mathbb{R}^{2})}
\end{align}
for some positive constant $\omega_0$. It suffices to claim that
\begin{align}\label{eq:3.47}
\left\|\left[\sum_{l\in \mathbb{Z}}\left|H_{u,\gamma,k+n_l(\cdot_1)}P_lf\right|^2\right]^{\frac{1}{2}}\right\|_{L^{p}(\mathbb{R}^{2})}\lesssim k^2\left\|f\right\|_{L^{p}(\mathbb{R}^{2})}
\end{align}
for all $p\in(1,\infty)$, since \eqref{eq:3.45} follows from interpolation between \eqref{eq:3.46} and \eqref{eq:3.47}.

We now consider $H_{u,\gamma,k+n_l(\cdot_1)}\mathbb{P}_lf$, where $\mathbb{P}_l$ is the same as \eqref{eq:y0}.
\begin{align}\label{eq:3.48}
&H_{u,\gamma,k+n_l(x_1)}\mathbb{P}_lf(x_1,x_2)\\
&\quad=  \int_{-\infty}^{\infty} \int_{-\infty}^{\infty} f(x_1-t,x_2-u(x_1)\gamma(t)-s) \frac{\psi_{k+n_l(x_1)}(t)}{t}\check{\rho_l}(s)\,\textrm{d}t\,\textrm{d}s\nonumber\\
 &\quad\leq \int_{\frac{1}{2}2^{k+n_l(x_1)}\leq|t|\leq 2\cdot 2^{k+n_l(x_1)}} \int_{-\infty}^{\infty}| f(x_1-t,x_2-u(x_1)\gamma(t)-s)|\left|\frac{\psi_{k+n_l(x_1)}(t)}{t} \right| |\check{\rho_l}(s)|\,\textrm{d}s\,\textrm{d}t\nonumber\\
 &\quad\lesssim  \frac{1}{2^{k+n_l(x_1)}}\int_{\frac{1}{2}2^{k+n_l(x_1)}\leq|t|\leq 2\cdot 2^{k+n_l(x_1)}} \int_{-\infty}^{\infty}| f(x_1-t,x_2-u(x_1)\gamma(t)-2^{-l}s)| |\check{\rho}(s)|\,\textrm{d}s\,\textrm{d}t\nonumber\\
&\quad\lesssim  \sum_{\tau\in \mathbb{Z}} \frac{1}{(1+|\tau|)^{4}}\frac{1}{2^{k+n_l(x_1)}}\nonumber\\
  &\quad\times\int_{\frac{1}{2}2^{k+n_l(x_1)}\leq|t|\leq 2\cdot 2^{k+n_l(x_1)}} \int_\tau^{\tau+1}| f(x_1-t,x_2-u(x_1)\gamma(t)-2^{-l}s)| \,\textrm{d}s\,\textrm{d}t\nonumber\\
 &\quad\approx \sum_{\tau\in \mathbb{Z}} \frac{1}{(1+|\tau|)^{4}}\frac{1}{2^{k+n_l(x_1)}}\nonumber\\
&\quad\times\int_{\frac{1}{2}2^{k+n_l(x_1)}\leq|t|\leq 2\cdot 2^{k+n_l(x_1)}} \int_0^{1}| f(x_1-t,x_2-u(x_1)\gamma(t)-2^{-l}(s+\tau))|\, \textrm{d}s\,\textrm{d}t.\nonumber
\end{align}
We want to control the last term in \eqref{eq:3.48} by
\begin{align}\label{eq:y1}
\sum_{\tau\in \mathbb{Z}} \frac{1}{(1+|\tau|)^{4}}\frac{1}{N_k}\sum_{m=0}^{N_k-1}\frac{1}{|I_m|}\int_{I_m} M_2^{(\sigma_m^{(2)})}f(x_1-t,x_2) \,\textrm{d}t,
\end{align}
where $\left\{I_m\right\}_{m=0}^{N_k-1}$ and the shifted maximal operator $M_2^{(\sigma_m^{(2)})}$ will be given soon. By a scaling argument, it suffices to prove that
\begin{align}\label{eq:y2}
&\sum_{\tau\in \mathbb{Z}} \frac{1}{(1+|\tau|)^{4}}\frac{1}{2^{k+n_l(x_1)}}\int_{\frac{1}{2}2^{k+n_l(x_1)}\leq|t|\leq 2\cdot 2^{k+n_l(x_1)}} \int_0^{1}| f(x_1-t,x_2-2^lu(x_1)\gamma(t)-s-\tau)| \,\textrm{d}s\,\textrm{d}t\\
&\quad \lesssim  \sum_{\tau\in \mathbb{Z}} \frac{1}{(1+|\tau|)^{4}}\frac{1}{N_k}\sum_{m=0}^{N_k-1}\frac{1}{|I_m|}\int_{I_m} M_2^{(\sigma_m^{(2)})}f(x_1-t,x_2) \,\textrm{d}t.\nonumber
\end{align}
We cover the region $\frac{1}{2}2^{k+n_l(x_1)}\leq|t|\leq 2\cdot2^{k+n_l(x_1)}$ by intervals $\left\{I_m\right\}_{m=0}^{N_k-1}$ where
$$I_m:=\left\{t\in \mathbb{R}:\,\frac{1}{2}2^{k+n_l(x_1)}+\frac{m}{2^l|u(x_1)| \gamma'(2^{k+n_l(x_1)})}\leq |t|\leq \frac{1}{2}2^{k+n_l(x_1)}+\frac{m+1}{2^l|u(x_1)| \gamma'(2^{k+n_l(x_1)})} \right\}$$
and $N_k\in \mathbb{N}$ is such that
\begin{align}\label{eq:3.49}
 \frac{3}{2}2^{k+n_l(x_1)}\leq  \frac{N_k}{2^l|u(x_1)| \gamma'(2^{k+n_l(x_1)})}\leq 2 \cdot2^{k+n_l(x_1)}.
\end{align}
Therefore,
\begin{align}\label{eq:3.50}
|I_m|=\frac{1}{2^l|u(x_1)|\gamma'(2^{k+n_l(x_1)})}
\end{align}
and
\begin{align}\label{eq:3.51}
\frac{1}{2}\frac{1}{2^{k+n_l(x_1)}}\leq \frac{1}{N_k\cdot|I_m|}\leq \frac{2}{3}\frac{1}{2^{k+n_l(x_1)}}.
\end{align}
Thus the first term in \eqref{eq:y2} can be controlled by
\begin{align}\label{eq:3.52}
 \sum_{\tau\in \mathbb{Z}} \frac{1}{(1+|\tau|)^{4}}\frac{1}{N_k}\sum_{m=0}^{N_k-1}\frac{1}{|I_m|}\int_{I_m} \int_0^{1}| f(x_1-t,x_2-2^lu(x_1)\gamma(t)-s-\tau)| \,\textrm{d}s\,\textrm{d}t.
\end{align}
No loss the generality, we denote
$${\Re}_m:=\left\{(t,2^lu(x_1)\gamma(t)+s+\tau)\in \mathbb{R}^2:\,t\in I_m, s\in (0,1) \right\}\in I_m\times J_m,$$
where
$$J_m:=\left[Ja,Jb\right],$$
$$Ja:=2^l|u(x_1)|\gamma\left(\frac{1}{2}2^{k+n_l(x_1)}+\frac{m}{2^l|u(x_1)| \gamma'(2^{k+n_l(x_1)})}\right)+\tau,$$
$$Jb:=2^l|u(x_1)|\gamma\left(\frac{1}{2}2^{k+n_l(x_1)}+\frac{m+1}{2^l|u(x_1)| \gamma'(2^{k+n_l(x_1)})}\right)+1+\tau.$$

We can show that
\begin{align}\label{eq:3.53}
|J_m|\approx 1.
\end{align}
In fact, by mean value theorem, it implies
\begin{align}\label{eq:3.54}
|J_m|
= 1+\frac{1}{ \gamma'(2^{k+n_l(x_1)})} \gamma'\left(\frac{1}{2}2^{k+n_l(x_1)}+\frac{m+\theta}{2^l|u(x_1)| \gamma'(2^{k+n_l(x_1)})}\right),
\end{align}
where $\theta\in [0,1]$.
It is easy to see that
\begin{align}\label{eq:3.55}
|J_m|\geq 1.
\end{align}
From $\theta\in [0,1]$ and \eqref{eq:3.49}, it follows that
$$m+\theta \leq N_k-1+\theta\leq N_k \leq 2 \cdot2^{k+n_l(x_1)}2^l|u(x_1)| \gamma'(2^{k+n_l(x_1)}).$$
Since $\gamma'$ is increasing on $(0,\infty)$ and $\frac{\gamma'(2t)}{\gamma'(t)}\leq C_1$ for any $t\in (0,\infty)$, we get
\begin{align}\label{eq:3.56}
|J_m|
\leq 1+\frac{\gamma'(4\cdot2^{k+n_l(x_1)})}{ \gamma'(2^{k+n_l(x_1)})}
= 1+\frac{\gamma'(4\cdot2^{k+n_l(x_1)})}{ \gamma'(2\cdot2^{k+n_l(x_1)})}\frac{\gamma'(2\cdot2^{k+n_l(x_1)})}{ \gamma'(2^{k+n_l(x_1)})}
\leq 1+C^2_1.
\end{align}
From \eqref{eq:3.55} and \eqref{eq:3.56}, it follows that \eqref{eq:3.53}.

Continuing the calculation in \eqref{eq:3.52}, which can be bounded by
\begin{align}\label{eq:3.57}
 \sum_{\tau\in \mathbb{Z}} \frac{1}{(1+|\tau|)^{4}} \frac{1}{N_k}\sum_{m=0}^{N_k-1}\frac{1}{|I_m|}\int_{I_m} \frac{1}{|J_m|}\int_{J_m} | f(x_1-t,x_2-s)|\,\textrm{d}s\, \textrm{d}t.
\end{align}
Given a non-negative parameter $\sigma$, the \emph{shifted maximal operator} is defined as
$$M^{(\sigma)}f(z):=\sup_{z\in I \subset\mathbb{R} }\frac{1}{|I|} \int_{I^{(\sigma)}}|f(\zeta)|\,\textrm{d}\zeta.$$
Here $I^{(\sigma)}$ denotes a shift of the interval $I:=[a,b]$ given by
$$I^{(\sigma)}:=[a-\sigma\cdot |I|,b-\sigma\cdot |I|]\cup [a+\sigma\cdot |I|,b+\sigma\cdot |I|].$$
We observe that
\begin{align}\label{eq:3.58}
 \frac{1}{|J_m|}\int_{J_m} | f(x_1-t,x_2-s)| \textrm{d}s\leq M_2^{(\sigma_m^{(2)})}f(x_1-t,x_2),
\end{align}
where $M_2^{(\sigma_m^{(2)})}$ is a shifted maximal operator applied in the second variable and
$$\sigma_m^{(2)}:=\frac{2^l|u(x_1)|}{|J_m|}\gamma\left(\frac{1}{2}2^{k+n_l(x_1)}+\frac{m}{2^l|u(x_1)| \gamma'(2^{k+n_l(x_1)})}\right)+\frac{\tau}{|J_m|}.$$
Combining \eqref{eq:3.57} and \eqref{eq:3.58}, we have established \eqref{eq:y2}.

Altogether we have now show that
\begin{align}\label{eq:3.60}
 |H_{u,\gamma,k+n_l(x_1)}\mathbb{P}_lf(x_1,x_2)|
 \lesssim \sum_{\tau\in \mathbb{Z}} \frac{1}{(1+|\tau|)^{4}}\frac{1}{N_k}\sum_{m=0}^{N_k-1}\frac{1}{|I_m|}\int_{I_m} M_2^{(\sigma_m^{(2)})}f(x_1-t,x_2) \,\textrm{d}t
\end{align}
for any $l\in \mathbb{Z}$. Therefore,
from \eqref{eq:y0},
\begin{align}\label{eq:3.62}
 |H_{u,\gamma,k+n_l(x_1)}P_lf(x_1,x_2)|
\lesssim  \sum_{\tau\in \mathbb{Z}} \frac{1}{(1+|\tau|)^{4}}\frac{1}{N_k}\sum_{m=0}^{N_k-1}\frac{1}{|I_m|}\int_{I_m} M_2^{(\sigma_m^{(2)})}P_lf(x_1-t,x_2) \,\textrm{d}t
\end{align}
for any $l\in \mathbb{Z}$.

Since $\gamma(0)=0$, similarly to Remark \ref{remark 1.1}, by the Cauchy mean value theorem, we have $\frac{\gamma(2t)}{\gamma(t)}\leq 2C_1$ holds for any $t\in (0,\infty)$. Notice that $\gamma$ is increasing on $(0,\infty)$, combining $m\leq N_k-1\leq N_k$, \eqref{eq:y3}, \eqref{eq:3.49}, \eqref{eq:3.53}, one obtains
\begin{align}\label{eq:3.59}
\sigma_m^{(2)}
\leq& \frac{1}{|J_m|}\frac{1}{\gamma (2^{n_l(x_1)})}\gamma\left(\frac{1}{2}2^{k+n_l(x_1)}+\frac{m}{2^l|u(x_1)| \gamma'(2^{k+n_l(x_1)})}\right)+\frac{\tau}{|J_m|}\\
\lesssim&\frac{1}{\gamma (2^{n_l(x_1)})}\gamma\left(\frac{1}{2}2^{k+n_l(x_1)}+\frac{2 \cdot2^{k+n_l(x_1)}2^l|u(x_1)| \gamma'(2^{k+n_l(x_1)}) }{2^l|u(x_1)| \gamma'(2^{k+n_l(x_1)})}\right)+ \tau\nonumber\\
\lesssim&\frac{\gamma (2^{n_l(x_1)+k+2})}{\gamma (2^{n_l(x_1)})}+ \tau
\lesssim (2C_1)^{k+2}+ \tau.\nonumber
\end{align}
From (\cite{G2}, Theorem 3.1), \eqref{eq:3.59} and the Littlewood-Paley theory, we get the following vector-valued estimate for the one-dimensional shifted maximal operator
\begin{align}\label{eq:3.63}
\left\|\left[\sum_{l\in \mathbb{Z}}\left|M_2^{(\sigma_m^{(2)})}P_lf(\cdot_1-t,\cdot_2)\right|^2\right]^{\frac{1}{2}}\right\|_{L^{p}(\mathbb{R}^1_{x_2})}
& \lesssim   \left[\log(2+|\sigma_m^{(2)}|)\right]^2 \left\|\left[\sum_{l\in \mathbb{Z}}\left|P_lf(\cdot_1-t,\cdot_2)\right|^2\right]^{\frac{1}{2}}\right\|_{L^{p}(\mathbb{R}^1_{x_2})}\\
& \lesssim  \left[\log(2+(2C_1)^{k+2}+|\tau|)\right]^2 \left\|f(\cdot_1-t,\cdot_2)\right\|_{L^{p}(\mathbb{R}^1_{x_2})}\nonumber\\
& \lesssim  k^2 (1+|\tau|)^2 \left\|f(\cdot_1-t,\cdot_2)\right\|_{L^{p}(\mathbb{R}^1_{x_2})}.\nonumber
\end{align}
Combining \eqref{eq:3.62}, by the triangle inequality and Minkowski's inequality, the left hand side of \eqref{eq:3.47} is controlled by
$$\sum_{\tau\in \mathbb{Z}} \frac{1}{(1+|\tau|)^{4}}
  \left\|\frac{1}{N_k}\sum_{m=0}^{N_k-1} \frac{1}{|I_m|}\int_{I_m} \left\| \left[\sum_{l\in \mathbb{Z}}\left| M_2^{(\sigma_m^{(2)})}P_lf(\cdot_1-t,\cdot_2) \right|^2\right]^{\frac{1}{2}}\right\|_{L^{p}(\mathbb{R}^1_{x_2})}\,\textrm{d}t\right\|_{L^{p}(\mathbb{R}^1_{x_1})} $$
Inserting \eqref{eq:3.63}, the above expression is bounded by
$$  k^2 \sum_{\tau\in \mathbb{Z}} \frac{1}{(1+|\tau|)^{2}}\left\|\frac{1}{N_k}\sum_{m=0}^{N_k-1} \frac{1}{|I_m|}\int_{I_m}\left\|f(\cdot_1-t,\cdot_2)\right\|_{L^{p}(\mathbb{R}^1_{x_2})}\,\textrm{d}t\right\|_{L^{p}(\mathbb{R}^1_{x_1})} $$
Noticing \eqref{eq:3.49} and \eqref{eq:3.51}, we can bound the above term by
\begin{align}
&  k^2 \sum_{\tau\in \mathbb{Z}} \frac{1}{(1+|\tau|)^{2}}\left\|\frac{1}{2^{k+n_l(\cdot_1)}}\int_{\frac{1}{2}2^{k+n_l(\cdot_1)}\leq|t|\leq
   \frac{5}{2}2^{k+n_l(\cdot_1)}}\left\|f(\cdot_1-t,\cdot_2)\right\|_{L^{p}(\mathbb{R}^1_{x_2})}\,\textrm{d}t\right\|_{L^{p}(\mathbb{R}^1_{x_1})}  \nonumber\\
  &\quad \lesssim  k^2 \sum_{\tau\in \mathbb{Z}} \frac{1}{(1+|\tau|)^{2}}\left\|M_1\left(\left\|f(\cdot,\cdot_2)\right\|_{L^{p}(\mathbb{R}^1_{x_2})}\right)(\cdot_1)\right\|_{L^{p}(\mathbb{R}^1_{x_1})}  \nonumber\\
 &\quad \lesssim  k^2 \sum_{\tau\in \mathbb{Z}} \frac{1}{(1+|\tau|)^{2}}\left\|\left\|f(\cdot_1,\cdot_2)\right\|_{L^{p}(\mathbb{R}^1_{x_2})}\right\|_{L^{p}(\mathbb{R}^1_{x_1})}
  \lesssim  k^2  \left\|f\right\|_{L^{p}(\mathbb{R}^2)}.\nonumber
\end{align}
Therefore, we obtain \eqref{eq:3.47}, which completes the proof of Theorem \ref{theorem 1.1}.
\end{proof}

\section*{Acknowledgements}

The authors would like to thank Prof. Dachun Yang for many valuable comments and discussions.

\bigskip

\noindent  Haixia Yu and Junfeng Li (Corresponding author)

\smallskip

\noindent  Laboratory of Mathematics and Complex Systems
(Ministry of Education of China),
School of Mathematical Sciences, Beijing Normal University,
Beijing 100875, People's Republic of China

\smallskip

\noindent {\it E-mails}: \texttt{yuhaixia@mail.bnu.edu.cn} (H. Yu)

\noindent\phantom{{\it E-mails:}} \texttt{lijunfeng@bnu.edu.cn} (J. Li)

\bigskip

\end{document}